\documentclass[a4paper]{amsart}
\input{article.sty}

\title[Elliptic Kato constant and isoperimetric inequalities]{Asymptotically non-negative Ricci curvature, elliptic Kato constant and isoperimetric inequalities}

\author[Debora Impera]{Debora Impera}
\address[Debora Impera]{Dipartimento di Scienze Matematiche ``Giuseppe Luigi Lagrange", Politecnico di Torino, Corso Duca degli Abruzzi, 24, Torino, Italy, I-10129}
\email{debora.impera@polito.it}

\author[Michele Rimoldi]{Michele Rimoldi}
\address[Michele Rimoldi]{Dipartimento di Scienze Matematiche ``Giuseppe Luigi Lagrange", Politecnico di Torino, Corso Duca degli Abruzzi, 24, Torino, Italy, I-10129}
\email{michele.rimoldi@polito.it}

\author[Giona Veronelli]{Giona Veronelli}
\address[Giona Veronelli]{Dipartimento di Matematica e Applicazioni, Universit\`a di Milano Bicocca, via R. Cozzi 53, I-20126 Milano, Italy}
\email{giona.veronelli@unimib.it}

\date{\today}

\begin{document}

\thanks{D. Impera and M. Rimoldi are partially supported by INdAM-GNSAGA G. Veronelli is partially supported by INdAM-GNAMPA. D. I. and M. R. acknowledge partial support by the PRIN 2022 project ``Real and Complex Manifolds: Geometry and Holomorphic Dynamics'' - 2022AP8HZ9. M.R and G.V. acknowledge partial support by the INdAM-GNAMPA project ``Applicazioni geometriche del metodo ABP'' - CUP: E55F22000270001.}



\begin{abstract}
The ABP method for proving isoperimetric inequalities has been first employed by Cabr\'e in $\R^n$, then developed by Brendle, notably in the context of non-compact Riemannian manifolds of non-negative Ricci curvature and positive asymptotic volume ratio. In this paper, we expand upon their approach and prove isoperimetric inequalities (sharp in the limit) in the presence of a small amount of negative curvature. First, we consider smallness of the negative part $\mathrm{Ric}_-$ of the Ricci curvature in terms of its elliptic Kato constant. Indeed, the Kato constant turns out to control the non-negativity of the ($\infty$-)Bakry-\'Emery Ricci-tensor of a suitable conformal deformation of the manifold, and the ABP method can be implemented in this setting. Secondly, we show that the smallness of the Kato constant is ensured provided that the asymptotic volume ratio is positive and either $M$ has one end and asymptotically non-negative sectional curvature, or there is a suitable polynomial decay of $\mathrm{Ric}_{-}$, and the relative volume comparison condition known as $\textbf{(VC)}$ holds. To show this latter fact, we enhance techniques elaborated by Li-Tam and Kasue to obtain new estimates of the Green function valid on the whole manifold.
\end{abstract}

\maketitle


\section{Introduction}
The famous isoperimetric problem consists in finding the region of fixed volume in a given underlying space which minimizes its perimeter. This problem has a quantitative counterpart known as isoperimetric inequality which, for homogeneity reasons, on $n$-dimensional spaces usually takes the form
\begin{equation}    \label{e:isop}
    \mathrm{Vol}(\Omega)^{(n-1)/n} \le C \mathrm{Area}(\partial \Omega)
\end{equation}
for some constant $C>0$, and for every region $\Omega$ in the space for which the volume and the area of the boundary are meaningful. 
In the Euclidean case, or more generally on a smooth Riemannian manifold $(M^n,g)$, it is well-known that the isoperimetric inequality has an equivalent functional formulation given by the $L^1$ Euclidean-type Sobolev inequality
\begin{equation}\label{e:L1Sob}
\|\varphi\|_{n/(n-1)}\le C \|\nabla \varphi\|_1,\qquad\forall\,\varphi\in C^\infty_c(M),
    \end{equation}
where the $L^p$-norms and the gradient are the canonical ones induced by the Riemannian structure, \cite{FF}. 

In the context of complete non-compact Riemannian manifolds, two main types of assumptions ensure the validity of \eqref{e:isop} and \eqref{e:L1Sob}. On the one hand, the isoperimetric inequality holds on Cartan-Hadamard (i.e. complete simply-connected with non-positive sectional curvature) manifolds, \cite{HS}, and the best possible value of the constant $C$ is predicted to be the Euclidean one by the so-called Cartan-Hadamard (or Aubin's) conjecture, so far proved to be true only for $n=2,3,4$, \cite{W,Kl,C}. On the other hand, for manifolds whose Ricci curvature is non-negative, N.T. Varopoulos \cite{Va} and G. Carron \cite{Ca-PhD} proved that \eqref{e:L1Sob} (and thus \eqref{e:isop}) holds if and only if $(M,g)$ is non-parabolic and the measures of the upper levels of the Green function are bounded from above by the Euclidean corresponding ones; see \cite[Theorem 8.4]{He} for more details. 
It is worth noticing that in this latter result the isoperimetric inequality is proved with a constant $C$ which is far from being the sharp one, e.g. in the Euclidean case. 

More recently, X. Cabr\'e \cite{Cab} introduced a new elegant strategy to approach the isoperimetric problem in $\R^n$ which exploits the so called ABP method (inspired by a celebrated quantitative maximum principle due to Alexandrov-Bakelmann-Pucci). This strategy is based on a deformation of the region $\Omega$, guided by the gradient of the solution of a Poisson problem on $\Omega$. Such a deformation turns out to decrease the isoperimetric ratio  $\mathrm{Area}(\partial \Omega)/\mathrm{Vol}(\Omega)^{(n-1)/n} $.
In the last years, the ABP approach has been perfected by S. Brendle in a series of breakthrough papers where, as special cases of his results, he shows the validity of the sharp isoperimetric inequality and the (Michael-Simon)-Sobolev inequality on (a) minimal submanifolds of $\R^n$ of codimension $1$ or $2$, \cite{brendle-jams}, and (b) complete manifolds of non-negative Ricci curvature and positive asymptotic ratio $0<\beta\doteq\lim_{r\to\infty} r^{-n}\mathrm{Vol}(B_r(x))$, or minimal submanifolds in ambient spaces with positive sectional curvature \cite{brendle}; see also \cite{AFM} for a previous different proof of (b) when $n=3$ and \cite{BK} for a subsequent alternative approach. 

\subsection{Isoperimetric inequalities beyond non-negative Ricci curvature}
The machinery introduced by Brendle has been subsequently adapted to a number of more general or related problems (see for instance \cite{brendle-cpam}, \cite{DLL-logsob}, \cite{MW}, \cite{Ph}, and also \cite{WZ} and \cite{XZ} for some previous related works).
In the context of smooth metric measure spaces, where the Riemannian measure is weighted by a non-negative smooth density $e^{-f}$ and the Ricci curvature is replaced by the $\alpha$--Bakry-\'Emery Ricci curvature \[
\mathrm{Ric}^\alpha_f\doteq\mathrm{Ric} + \mathrm{Hess} f - \alpha^{-1} df\otimes df,\]
$\alpha<\infty$, first F. Johne \cite{Johne-arxiv} proved the $L^1$-Sobolev inequality
\[
\|\varphi\|_{L^{\frac{n+\alpha}{n+\alpha-1}}(\Omega,e^{-f}d\mathrm{vol})}(n+\alpha)\beta_\alpha^{1/(n+\alpha)}\le\int_\Omega |\nabla \varphi| e^{-f} d\mathrm{vol} + \int_{\partial\Omega}  \varphi e^{-f} d\mathrm{vol}_{n-1} 
\]
for all smooth positive function $\varphi$ on $\bar\Omega$, $\beta_\alpha\doteq\limsup_{r\to\infty}  r^{-(n+\alpha)}\mathrm{vol}_f(B_r(x))$ being the weighted volume asymptotic ratio, where $vol_{f}(B_{r}(x))=\int_{B_{r}(x)}e^{-f}d\mathrm{vol}$.
Y. Dong, H. Lin and L. Lu managed to allow a small amount of negative curvature both in the unweighted \cite{DLL-unweight} and in the weigthed \cite{DLL-weight} settings, considering manifolds with asymptotically non-negative Ricci curvature, i.e. such that for some fixed reference point $o\in M$,
\begin{equation}\label{asymp nonneg}
\mathrm{Ric}(q)\ge -(n-1)\lambda(d(q,o))g(q)\qquad \text{( resp. }\mathrm{Ric}^\alpha_f(q)\ge -(n-1)\lambda(d(q,o))g(q)\text{ )},\end{equation}
for all $q\in M$ in the sense of quadratic forms, where $\lambda:[0,\infty)\to[0,\infty)$ is a non-increasing function  such that
\begin{equation}   \label{e:asymp curv}
    b_0\doteq\int_0^\infty s\,\lambda(s)\,ds<+\infty,\end{equation} 
as well as submanifolds in spaces of asymptotically non-negative sectional curvature. Note that for \eqref{e:asymp curv} to be satisfied, necessarily $\lambda(t)=o(t^{-2})$. In a sense, these curvature assumptions are almost the best one can hope for, as one can construct examples of manifolds with $\beta>0$ and $\lambda(s) = (1+s)^{\epsilon-2}$, $\epsilon>0$, for which the isoperimetric inequality fails to hold. However, from the perspective of isoperimetric inequalities, the $L^1$-Sobolev inequality in \cite{DLL-unweight,DLL-weight} is somehow unsatisfactory as it contains an additional additive term of the form $\|\varphi\|_{L^1(\Omega)}$ at the RHS of \eqref{e:L1Sob} and, moreover, its validity is limited to domains contained in a bounded region.
One of the outcomes of the present paper consists in removing these flaws. The price to pay is the requirement of either a stronger, but still polynomial, decay of the curvature at infinity, or a stronger assumption on the sectional (instead of Ricci) curvature.
\medskip

To introduce our setting let us recall that, besides asymptotic non-negativity, there exist different other notions of smallness of the negative part of the curvature. In the last years, a particular interest has been received by the so-called Kato conditions for the Ricci curvature.
These are conditions inspired by the Kato class appearing in T. Kato's investigations of the self-adjointness of Schr\"odinger operators with potentials which change sign. Namely, let
\[
k_T(\mathrm{Ric}_-)=\sup_{x\in M}\int_0^T\int_M  h_t(x,y)\mathrm{Ric}_-(y)\,dt\,d\mathrm{vol}(y)
,
\]
where we are denoting by $h_t(x,y)$ the heat kernel of $(M,g)$ and by $\mathrm{Ric}_- : M \to [0,\infty)$ the lowest non-negative function such that
\[
\mathrm{Ric}(x) \ge -\mathrm{Ric}_-(x)g(x)
\]
in the sense of quadratic forms for any $x\in M$. One says that the Ricci curvature of $M$ is (a) in the Kato class if $k_T$ vanishes in the limit as $T\to 0$ and (b) in the Dyinkin class if $k_T<1$ for some $T>0$. To the best of our knowledge, the Kato's type conditions for the negative part of the curvature were first considered in \cite{GP}. Since then, several results which were well-known in the setting of non-negative Ricci curvature have been generalized to manifolds whose Ricci curvature is in the Kato or in the Dynkin class; see for instance;
 \cite{Ca, CMT1, CMT2, CMT3, CR, GK, Ro, Le, RS, RW}.
In this paper, we will consider a stronger  condition by requiring instead that $k_\infty(\mathrm{Ric}_-)$ is small enough. Such an assumption on a complete manifold was first considered in \cite{Devyver1}; see also \cite{Ca,devyver, CDS} for subsequent developments.
Following the terminology of \cite{Ca}, we will call $k_\infty(\mathrm{Ric}_-)$ the \textsl{elliptic Kato constant} of $M$. Note that, by Fubini's theorem, the elliptic Kato constant $k_\infty(\mathrm{Ric}_-)$ can be written as
\begin{equation}\label{e:kinfty}
k_\infty\doteq k_\infty(\mathrm{Ric}_-)=\sup_{x\in M}\int_M G(x,y)\mathrm{Ric}_-(y)\,d\mathrm{vol}(y),
\end{equation}
as soon as $(M,g)$ possesses a positive Green kernel $G$, i.e. is non-parabolic. Our first main result is the following
\begin{theorem}\label{isoKato}
Let $(M^n, g)$, $n\geq3$, be a complete non-parabolic Riemannian manifold satisfying
\begin{equation}\label{e:AVR 1.1}
\limsup_{r\to\infty}\frac{\mathrm{vol}(B_r(y))}{r^n}\geq \beta>0,
\end{equation}
for some (hence any) $y\in M$. Assume that
$k_\infty <  \frac{1}{n-2}$. Let $\Omega$ be a compact domain with smooth boundary $\partial \Omega$ and let $h$ be a positive smooth function on $\Omega$. Then
\begin{align*}
\left(\int_\Omega \ h^{\frac{n}{n-1}}d\mathrm{vol}\right)^{\frac{n-1}{n}}&\le C(k_\infty,n,\beta)^{-1}\left(\int_\Omega |\nabla h|\,d\mathrm{vol}
+ \int_{\partial \Omega} h\,d\mathrm{vol}_{n-1}\right),
\end{align*}
where $C(k_\infty,n,\beta)=n \left(1-(n-2)k_\infty\right)^{\frac{4(n-1)}{n(n-2)}}\beta^{\frac{1}{n}}$. In particular, the following isoperimetric inequality holds on $M$
\[
\mathrm{Vol}(\Omega)^{\frac{n-1}{n}}\le C(k_\infty,n,\beta)^{-1}\mathrm{Area}(\partial \Omega).
\]
\end{theorem}

If the $\limsup$ in the asymptotic volume ratio assumption is replaced by a stronger $\liminf$, then the non-parabolicity is automatically granted by standard volume conditions for parabolicity, see e.g.  \cite{Gr}. In any case, asking for non-parabolicity is not so restrictive, since $k_\infty$ makes no sense on parabolic manifolds.

 It is worth noticing that, in the limit as $k_\infty\to0$,  the constant is the Euclidean one, hence the sharp one of manifolds with non-negative Ricci curvature. The validity of an isoperimetric inequality on manifolds with small elliptic Kato constant and positive isoperimetric ratio can alternatively deduced by combining \cite[Proposition 4.6]{CaMT} with \cite[Theorem 3]{CS-C}. However, following this alternative approach the resulting constant that one finds is far from being sharp.

In order to prove Theorem \ref{isoKato} we will use the ABP approach by Cabr\'e-Brendle. To the best of our knowledge, this is the first time that this approach is used to obtain a genuine isoperimetric inequality (i.e., uniform on the whole space and without additional integral terms) in the presence of some negative curvature.

\subsection{About the smallness of the Elliptic Kato constant}
The assumption $k_\infty \le (n-2)^{-1}$ is clearly quite strong, since it implies the existence of a bounded positive global solution to $-\Delta\phi = \mathrm{Ric}_-$ on $M$. However, there exist significant classes of manifolds whose Ricci curvature attains also negative values which turn out to have small $k_\infty$, and hence to which our result applies. Notably, we have the following theorem. Recall that $M$ is said to satisfy condition \textbf{(VC)} with respect to a pole $o\in M$ if there exists a constant $\xi>0$ such that for all $r$ and $x\in \partial B_r(o)$, $\mathrm{vol}(B_r(o))\leq \xi \mathrm{vol}(B_{r/2}(x))$; \cite{LiTam95}. Moreover $M$ is said to have asymptotically non-negative sectional curvature if $\mathrm{Sect}_x\ge -\lambda(d(x,o))
    $, with $b_0=\int_0^\infty s\lambda(s)\,ds<\infty$, for some non-increasing function $\lambda:[0,\infty)\to[0,\infty)$.

\begin{theorem}\label{coro:asymp}
Let $(M^n, g)$, $n\geq3$, be a complete Riemannian manifold satisfying
\[
\limsup_{r\to\infty}\frac{\mathrm{vol}(B_r(q))}{r^n}\geq \beta>0
\]
for some (hence any) $q\in M$ and some $\beta\in (0,\infty)$. Suppose also that either:
\begin{itemize}
    \item[(a)] for some fixed reference point $o\in M$, for some $K\ge 0$ and $\alpha>3n$,
    \[
\mathrm{Ric}(x)\geq -\frac{(n-1)K}{1+d(o,x)^\alpha}g(x),\qquad\forall\, x\in M\] and $M$ satisfies the condition $\textbf{(VC)}$ with constant $\xi>0$ with respect to $o$, or
    \item [(b)] $M$ has only one end and asymptotically non-negative sectional curvature. 
\end{itemize}
Then $k_\infty(\operatorname{Ric}_-)\le C_1$.

Moreover, there exist constants $\tilde K=\tilde K( n,\alpha,\beta,\xi)$ and $\tilde b_0=\tilde b_0(n,\beta)$ such that if $K\le \tilde K$ in case (a) or $b_0\le \tilde b_0$ in case (b), then
\begin{align*}
\left(\int_\Omega \ h^{\frac{n}{n-1}}d\mathrm{vol}\right)^{\frac{n-1}{n}}&\le C_2^{-1}\left(\int_\Omega |\nabla h|\,d\mathrm{vol}
+ \int_{\partial \Omega} h\,d\mathrm{vol}_{n-1}\right),
\end{align*}
for every compact domain $\Omega$ with smooth boundary $\partial \Omega$ and every positive smooth function $h$ on $\Omega$.
In particular, the following isoperimetric inequality holds on $M$
\[
\mathrm{Vol}(\Omega)^{\frac{n-1}{n}}\le C_2^{-1}\mathrm{Area}(\partial \Omega),
\]
and $C_2 \to n \beta^{\frac{1}{n}}$ when $K\to 0$. Here, the constants $C_1$ and $C_2$ depend on $n,\beta,\alpha,K,\xi$ in the assumption (a), and on $n,\beta,b_0$ in the assumption (b).
\end{theorem}

Note that the constant $C_1$ bounding from above $k_{\infty}(\mathrm{Ric}_{-})$ could be in principle written explicitly, but its expression is quite involved. On the other hand, in order to deduce the second assertion in Theorem \ref{coro:asymp} it will be enough to ensure that this constant goes to $0$ as $K\to 0$.

It is worth noticing that in case $M$ has asymptotically non-negative sectional curvature, and the asymptotic volume ratio $\beta$ is large enough with respect to the curvature bound, then Bazanfar\'e proved in \cite{Baz} that $M$ must be diffeomorphic to $\mathbb{R}^n$, and so in particular it has only one end. 

As observed, in the assumptions of Theorem \ref{coro:asymp}, the isoperimetric inequality \eqref{e:isop} holds on $M$. We are not aware of any previously known result in the literature in this direction. Moreover, standard techniques employed to prove Euclidean Sobolev inequalities prior to Brendle contribution did not give sharp constant already in the $\mathrm{Ric}\ge 0$ setting. Instead, our proof
gives a sharp constant in the limit as $K\to 0$. Concerning the threshold polynomial order $3n$, we specify that this is a technical assumption due to our method of proof. One naturally wonders if Theorem \ref{coro:asymp}(a) may hold true for any $\alpha>2$.

Let us present here briefly the main ideas behind the proof of Theorem \ref{coro:asymp}. A more detailed discussion about the difficulties and the main novelties in the proof will be given in Section \ref{sec:asymp}. 
As it is clear when looking at the expression in \eqref{e:kinfty}, the upper bound on $k_\infty$ in Theorem \ref{coro:asymp} can be obtained through a control of $\mathrm{Ric}_-$ combined with suitable estimates for the Green kernel $G(x,y)$. When $\mathrm{Ric}\geq 0$, a celebrated result by Li and Yau \cite{LY} gives that
\begin{equation}\label{e:GreenEstimates_intro}
G(x,y)\leq C(n)\int_{d(x,y)}^{+\infty} \frac{t}{\mathrm{vol}(B_t(y))}dt,\qquad\forall\,x,y\in M,
\end{equation} 
with $C(n)$ a purely dimensional constant. In the presence of a lower bound on the volume of balls, this yields an explicit estimate for $G$. The combined efforts of Li-Tam and Kasue, \cite{LiTam-AJM,Ka,LiTam95}, permitted to generalize \eqref{e:GreenEstimates_intro} to manifolds which satisfy \begin{equation}\label{e:curv litam}
\mathrm{Ric}(x)\geq -\frac{(n-1)K}{1+d(o,x)^2}g(x),\qquad\forall\, x\in M\end{equation}
(hence in particular to asymptotically non-negative Ricci curvature as in 
\eqref{asymp nonneg})
plus the \textbf{(VC)} relative volume comparison condition introduced above, provided that either $x$ or $y$ coincide with the reference point $o\in M$ with respect to which the curvature condition is considered. Obviously, the lower bound \eqref{e:curv litam} can not be satisfied with the same $K$ for any reference point $p\in M$ (unless $\mathrm{Ric}\ge 0$). In fact, when \eqref{e:curv litam} is satisfied with a given $K$ with respect to $o$, then the same relation holds with respect to a new reference point $p$ (replacing $o$) with a different value of $K$ which depends quadratically in $d(o,p)$; see \eqref{e:K tilde}. Since the constant in \eqref{e:GreenEstimates_intro} depends doubly exponentially on $K$, \cite{Fa}, the resulting estimate is unsatisfactory, as it permits to deal only with manifolds whose negative part of the curvature decays (doubly) exponentially. To overcome this problem, we refine the techniques introduced by Li-Tam and prove estimates for the Green function $G(p,x)$ centered at a point $p$ with respect to which the following curvature condition holds
\begin{equation*}
\mathrm{Ric}(x)\geq -\frac{(n-1)K}{1+(d(p,x)-d(o,p))^2}g(x),\qquad\forall\, x\in M.\end{equation*} 
This, in turn, is straightforwardly implied by \eqref{e:curv litam}. The main difficulty here lies in dealing with a "bad" annular region about $\partial B_{d(o,p)}(p)$, which gives rise to a dependence of the constant $C$ in \eqref{e:GreenEstimates_intro} on $d(o,p)$. However, such a dependence turns out to be polynomial and no more exponential.

It is worth noticing that in the presence of a lower bound for the sectional curvatures, such as in the assumption (b), the proof of Theorem \ref{coro:asymp} simplifies, as 
in that setting uniform (i.e. independent of the distance from $o$) estimates for the $G(p,x)$ are already available in the literature; see \cite{GSC} and Remark \ref{rmk:green estimates} below.

\subsection{A key step: isoperimetry in the weighted setting via ABP}

One of the main features of the smallness of $k_\infty$, which we will also exploit in our work, is that it implies the gaugeability of the Schr\"odinger operator $-\Delta - (n-2) \mathrm{Ric}_-$. Following an approach inspired by \cite{CaMT} (where parabolic Kato constants $k_T$ were considered), this gaugeability implies that the weighted manifold $(M^n,\widetilde{g}\doteq e^{2f}g, e^{(2-n)f}\,d\mathrm{vol}_{\widetilde g}=e^{2f}\,d\mathrm{vol}_g)$ has non-negative ($\infty$-)Bakry-\'Emery Ricci tensor, i.e. 
\[
\mathrm{\widetilde{Ric}}_{(n-2)f}\doteq \mathrm{Ric}+(n-2)\mathrm{Hess}f\geq 0.
\]
Here $f$ is a bounded function, whose bound depends on $k_\infty$. In particular, weighted and unweighted  volumes and norms are comparable. Note that this operation of taking a conformal deformation and suitably weighting the volume measure corresponds to a time change for the Brownian motion on $(M, g)$; \cite[Remark 1]{HanSturm}. As a consequence, proving Theorem \ref{isoKato} boils down to prove the following weighted version of Brendle's isoperimetry of independent interest. 

\begin{theorem}\label{hbounded}
Let $M^{n}_f$ be a complete non-compact weighted manifold with $\mathrm{Ric}_{f}\geq 0$. Assume that $\left\|f\right\|_{\infty}\leq k<\infty$ and 
\begin{equation}\label{f-AVR} \limsup_{r\to\infty}\frac{\mathrm{vol}_{f}(B_{r}(y))}{r^{n}}\geq \beta>0,
\end{equation}
for some (hence any) $y\in M$. Let $\Omega$ be a compact domain with smooth boundary $\partial \Omega$ and let $h$ be a positive smooth function on $\Omega$. 
 Then
\begin{equation*}
\left(\int_\Omega \ h^{\frac{n}{n-1}}e^{-f}d\mathrm{vol}\right)^{\frac{n-1}{n}}\le\frac{e^{\frac{4k}n}}{n \beta^{\frac{1}{n}}}\left(\int_\Omega |\nabla h|e^{-f}\,d\mathrm{vol} + \int_{\partial \Omega} h e^{-f}\,d\mathrm{vol}_{n-1}\right).
\end{equation*}
\end{theorem}

\begin{remark}
Note that the $\limsup$ in \eqref{f-AVR} is finite under our assumptions; see \cite[Lemma 2.1]{MunteanuWang} and also \cite{Yang}. As a consequence of the proof of Theorem \ref{isoKato}, also the $\limsup$ in \eqref{e:AVR 1.1}  is finite.
\end{remark}

Our proof of Theorem \ref{hbounded} follows the original strategy introduced by Brendle. Namely, as in \cite{Johne-arxiv} we consider, after scaling, the solution of the Neumann problem
\[
\begin{cases}
e^{f}\mathrm{div}_f(h\nabla u)= nh^{\frac{n}{n-1}}-|\nabla h|,&\textrm{in }\Omega\\
\partial_\nu u = 1,& \text{on }\partial\Omega,
\end{cases}
\]
and define a deformation map
\[
\Phi_r(x)=\exp_x(r\nabla u(x)).
\]
This map can be proved to be surjective onto large balls of radius comparable to $r$ and turns out to decrease the weighted isoperimetric ratio. This latter fact is due to an estimate of the Jacobian determinant of $\Phi_r$. If $\mathrm{Ric}^\alpha_f$ is non-negative one can retrace Brendle's proof by adding a virtual dimension $\alpha$ to the space \cite{Johne-arxiv}. When $\alpha=\infty$ we need to be more careful as the differential equation satisfied by (a function of) the $|D\Phi_r|$ is a bit more involved. To obtain a profitable estimate, we thus need to require the boundedness of the weight, as it is customary in the literature (see e.g. \cite{WW}).   \\

\subsection{Structure of the paper} The paper is organized as follows. In Section \ref{sec:weighted} we will prove the weighted isoperimetric inequality contained in Theorem \ref{hbounded}. In Section \ref{sec:Kato} we will recall some basic facts about the elliptic Kato constant, and will prove Theorem \ref{isoKato}. In Section \ref{sec:asymp} we will provide conditions on manifolds with asymptotically non-negative curvature which ensure the smallness of the Kato constant. In particular, this will permit to deduce the validity of Theorem \ref{coro:asymp}. Its proof is based on fine estimates of the Green kernel, stated in Theorem \ref{th:green estimate}, and proved in the subsequent Section \ref{Sect:GreenEstimates} as a special case of the more abstract Theorem \ref{th:green estimate_abstract}.

\subsection*{Convention.} We use the convention according to which $C(\,\cdot\,,\,\cdot\,,\,\dots)$ always denotes a constant whose value can change at each occurrence, but which depends (in principle explicitly) only on the quantities indicated in brackets.

\subsection*{Acknowledgements.} We are indebted to O. Munteanu for pointing out an error in a previous version of this manuscript (see Remark \ref{rmk:bazan}). We would like to thank the anonymous referee for their thoughtful and constructive comments, which significantly improved the quality of this manuscript. In particular, thanks to their suggestion, the assumption in case (b) of Theorem \ref{coro:asymp} has been weakened. We are also grateful to G. Carron, B. Devyver, B. G\"uneysu and D. Tewodrose for useful suggestions and clarifications, especially about the Kato condition, which improved the presentation. Finally, we would like to thank S. Pigola for several fruitful conversations.

\section{Proof of Theorem \ref{hbounded}}\label{sec:weighted}

We start following the approach in \cite{Johne-arxiv}. We may assume by scaling that
\begin{equation}\label{e:scaling}   
\int_\Omega |\nabla h|e^{-f}\,d\mathrm{vol} + \int_{\partial \Omega} h e^{-f}\,d\mathrm{vol}_{n-1} = n \int_\Omega h^{\frac{n}{n-1}}e^{-f}\,d\mathrm{vol}.
\end{equation}
Consider the Neumann problem
\[
\begin{cases}
e^{f}\mathrm{div}(e^{-f}h\nabla u)= nh^{\frac{n}{n-1}}-|\nabla h|,&\textrm{in }\Omega\\
\partial_\nu u = 1,& \text{on }\partial\Omega.
\end{cases}
\]
The assumption \eqref{e:scaling}
ensures that the Neumann problem has a solution. Since $h$ is
smooth, we have $|\nabla h| \in C^{0,1}$ and hence by standard elliptic theory (see for example Theorem 6.31
in \cite{GT}) we conclude $u \in C^{2,\alpha}$
 for any $\alpha\in (0,1)$.
 Define
 \[
 U=\{x\in\Omega\setminus\partial\Omega\,:\,|\nabla u(x)|<1\}
 \]
 and for $r>0$
 \[
A_r=\{x\in U\,:\,\forall y\in U,\ ru(y)+\frac{1}{2}d(y,\exp_x(r\nabla u(x))^2 \ge ru(x) +\frac{1}{2}r^2|\nabla u(x)|^2\} 
 \]
 Define the $C^{1,\alpha}$
transport map $\Phi_r:\Omega\to M$ by
\[
\Phi_r(x)=\exp_x(r\nabla u(x)).
\]
We have the following
\begin{lemma}[cf. Lemma 2.1 in \cite{Johne-arxiv}]\label{lem:2.1}
Let $x\in U$. Then
\[
\Delta_fu(x)\doteq e^{f}\mathrm{div}(e^{-f}\nabla u)(x)\le n h(x)^{\frac{1}{n-1}}.
\]
\end{lemma}
\begin{proof}
\[
\Delta_f u= n h^{\frac{1}{n-1}} - \frac{|\nabla h|}{h}-\frac{\langle\nabla h,\nabla u\rangle}{h}\le n h^{\frac{1}{n-1}}.
\]
\end{proof}
\begin{lemma}[Lemma 2.2 in \cite{brendle}]\label{lem:2.2}
The set
\[
\{p \in M\,:\, d_g(x, p) < r\text{ for all }x \in \Omega\}
\]
is contained in the image $\Phi_r(A_r)$ of $A_r$ through the transport map. In particular, for $r>\diam\Omega$ and for any $y\in\Omega$, $B_{r-\diam\Omega}(y)\subset \Phi_r(A_r)$.
\end{lemma}

We now prove the following
\begin{proposition}
Assume that $\bar x\in A_r$. Then 
\[
|\mathrm{det}D\Phi_r|(\bar x)\le
e^{f(\bar\gamma(r))-f(\bar x)}\left[ 1 +h(\bar x)^{\frac{1}{n-1}}\int_0^re^{\frac{2}{n}(f(\bar\gamma(s))-f(\bar x))}\,ds\right]^n,
\]
where $\bar\gamma
 : [0, r] \to M$ is defined by $\bar \gamma
(t) = \exp_{\bar x}(t\nabla u(\bar x))$. 
\end{proposition}
\begin{proof}
Choose an orthonormal basis $\{e_1, \dots , e_m\}$
of the tangent space $T_{\bar x}M$, and construct geodesic normal coordinates $(x^1, \dots , x^m)$ around $\bar x$, such
that we have $\partial_i = e_i$ at $\bar x$. 
We construct for $1 \le i \le m$ vector fields $E_i$ along $\bar\gamma$
 by parallel transport of the vector fields $e_i$.
Moreover, we solve the Jacobi equation to obtain the unique Jacobi fields $X_i$ along $\bar\gamma$
 satisfying
$X_i(0) = e_i$ and
\[\langle D_tX_i(0), e_j\rangle = (\nabla^2u)(\bar x)(e_i, e_j),\]
where $D_t$ denotes the covariant derivative along $\bar\gamma$.

Let us define a matrix-valued function $P : [0, r ] \to \mathbb R^{n\times n}$ by
\[
[P(t)]_{ij} = \langle X_i(t),E_j (t)\rangle.
\]
In particular, $P(0)=I_n$ and $P'(0)=\nabla^2 u(\bar{x})$ and $\det P(r)=|\mathrm{det} D\Phi_r|(\bar x)$.
Define 
\[
Q(t)=P(t)^{-1}P'(t).
\]
According to Jacobi's formula, $\operatorname{tr} Q(t)=\frac{d}{dt}\log \det P(t)$. Moreover, computing as in the proof of \cite[Lemma 2.5]{Johne-arxiv} we obtain the differential relation
\begin{equation*}
\frac{d}{dt}\operatorname{tr} Q(t)= - \mathrm{Ric}(\bar\gamma'(t),\bar\gamma'(t))-\operatorname{tr}[Q(t)^2].
\end{equation*}
By our curvature assumptions and the Cauchy-Schwarz inequality this latter yields
\begin{equation}\label{e:trQ}
\frac{d}{dt}\left(\operatorname{tr} Q(t)-F'(t)\right)\le - \frac{1}{n}(\operatorname{tr} Q(t))^2,
\end{equation}
where $F(t)=f(\bar\gamma(t))$.
Set 
\[
y(t)=[e^{-F(t)}\det P(t)]^{\frac{1}{n}}.
\]
 Then
\begin{align}\label{e:y'}
y'(t)&=\frac{1}{n}y(t)\frac{d}{dt}\log [e^{-F(t)}\det P(t)]\\
&=\frac{1}{n}y(t)[\operatorname{tr} Q(t)-F'(t)],\nonumber
\end{align}
which implies
\begin{equation}\label{e:trQlhs}
\frac{d}{dt}[\operatorname{tr} Q(t)-F'(t)]= n\frac{y''(t)}{y(t)}-n\left(\frac{y'(t)}{y(t)}\right)^2.
\end{equation}
Moreover,
\begin{equation}\label{e:trQrhs}
- \frac{1}{n}(\operatorname{tr} Q(t))^2 = - \frac{1}{n}\left(n\frac{y'(t)}{y(t)}+F'(t)\right)^2=-n\left(\frac{y'(t)}{y(t)}\right)^2 - 2 F'(t)\frac{y'(t)}{y(t)}-\frac{1}{n}(F'(t))^2.
\end{equation}
Inserting \eqref{e:trQlhs} and \eqref{e:trQrhs} into \eqref{e:trQ} gives
\begin{equation*}
    n\frac{y''(t)}{y(t)} \le -2 F'(t)\frac{y'(t)}{y(t)}-\frac{1}{n}(F'(t))^2 \le -2 F'(t)\frac{y'(t)}{y(t)},
\end{equation*}
i.e.
\[
ny''(t)\le -2F'(t)y'(t).
\]
Integrating this latter we obtain
\[
y'(t)\le y'(0) e^{-\frac{2}{n}(F(t)-F(0))}.
\]
A further integration yields
\[
y(t)\le y(0)+y'(0) \int_0^te^{-\frac{2}{n}(F(s)-F(0))}\,ds.
\]
We have $y(0)=e^{-\frac{f(\bar x)}{n}}$ and recalling \eqref{e:y'}
\begin{align*}
    y'(0)&=\frac{1}{n}e^{-\frac{f(\bar x)}{n}}\left(\operatorname{tr}P^{-1}(0)P'(0)-F'(0)\right)\\
    &=\frac{1}{n}e^{-\frac{f(\bar x)}{n}}\left(\Delta u(\bar x)- \langle \nabla f(\bar x),\bar\gamma'(0)\rangle\right) = \frac{1}{n}e^{-\frac{f(\bar x)}{n}}\Delta_f u(\bar x).
\end{align*}
Hence, using also Lemma \ref{lem:2.1},
\begin{align}\label{e:Dphi}
|\mathrm{det} D\Phi_r|(\bar x) = \det P(r) = e^{{F(r)}}y(r)^n
&\le e^{f(\bar\gamma(r))-f(\bar x)}\left[ 1 +\frac{\Delta_fu(\bar x)}{n}\int_0^re^{-\frac{2}{n}(f(\bar\gamma(s))-f(\bar x))}\,ds\right]^n\\
&\le e^{f(\bar\gamma(r))-f(\bar x)}\left[ 1 +h(\bar x)^{\frac{1}{n-1}}\int_0^re^{-\frac{2}{n}(f(\bar\gamma(s))-f(\bar x))}\,ds\right]^n.\nonumber
\end{align}
\end{proof}

Assume now that $\|f\|_\infty=k<\infty$.
Then
\begin{align*}
e^{-f\circ\Phi_r(\bar x)}|\mathrm{det} D\Phi_r|(\bar x) 
&\le e^{-f(\bar x)} e^{4k}\left[ \frac{e^{-\frac{4k}{n}}}{r} +h(\bar x)^{\frac{1}{n-1}}\right]^n r^n.
\end{align*}
For $y\in \Omega$, by Lemma \ref{lem:2.2} and the area formula, 
\begin{align*}
\mathrm{vol}_f(B_{r-\diam\Omega}(y)) &\le  \int_{\Phi_r(A_r)} e^{-f} d\mathrm{vol}\le \int_{A_r} |\det D\Phi_r| e^{-f\circ\Phi_r}d\mathrm{vol}
\\
&\le \int_\Omega e^{4k}\left[ \frac{e^{-\frac{4k}n}}{r} +h^{\frac{1}{n-1}}\right]^n r^n e^{-f}d\mathrm{vol}.
\end{align*}
Suppose that 
\[
\limsup_{r\to\infty}\frac{\mathrm{vol}_f(B_{r}(y)) }{r^n}\ge \beta >0.
\]
Then
\begin{align*}
    \beta &\le \limsup_{r\to\infty}\frac{\mathrm{vol}_f(B_{r-\diam\Omega}(y)) }{r^n}
    \le \limsup_{r\to\infty}\int_\Omega e^{4k}\left[ \frac{e^{-\frac{4k}n}}{r} +h^{\frac{1}{n-1}}\right]^n e^{-f}d\mathrm{vol}\\
    &=e^{4k}\int_\Omega \ h^{\frac{n}{n-1}}e^{-f}d\mathrm{vol}
\end{align*}
from which
\begin{align*}
\beta^{\frac{1}{n}}\left(\int_\Omega \ h^{\frac{n}{n-1}}e^{-f}d\mathrm{vol}\right)^{\frac{n-1}{n}}&\le e^{\frac{4k}n} \int_\Omega \ h^{\frac{n}{n-1}}e^{-f}d\mathrm{vol} \\
&=\frac{e^{\frac{4k}n}}{n}\left(\int_\Omega |\nabla h|e^{-f}\,d\mathrm{vol} + \int_{\partial \Omega} h e^{-f}\,d\mathrm{vol}_{n-1}\right),
\end{align*}
the last equality being due to the scaling assumption \eqref{e:scaling}.

This concludes the proof of Theorem \ref{hbounded}.

\begin{remark}\label{rmk:k}
    An inspection of the proof shows that the constant $k$ in the statement can be improved to 
    \[k=\frac{1}{2}(\sup_M f - \inf_M f).\]
\end{remark}

\subsection{A remark about the case $|\nabla f|$ bounded}
One is led to wonder if this technique adapts also to the case $|\nabla f|$ bounded. The following discussion goes in the direction of a negative answer.

Assume that $0<\|\nabla f\|_\infty=a<\infty$.
As $|\nabla u|<1$ on $A_r$, we have 
\[
e^{|f(\bar\gamma(s))-f(\bar x)|}\le e^{sa},
\]
from which
\begin{align*}
|\mathrm{det} D\Phi_r|(\bar x) 
&\le e^{f(\bar\gamma(r))-f(\bar x)}\left[ 1 +h(\bar x)^{\frac{1}{n-1}}\int_0^re^{-\frac{2}{n}(f(\bar\gamma(s))-f(\bar x))}\,ds\right]^n\\
&= e^{f(\bar\gamma(r))-f(\bar x)}\left[ 1 +h(\bar x)^{\frac{1}{n-1}}\int_0^re^{\frac{2}{n}as}\,ds\right]^n\\
&\le e^{f(\bar\gamma(r))-f(\bar x)}\left[ 1 +\frac{n}{2a}h(\bar x)^{\frac{1}{n-1}}e^{\frac{2}{n}ar}\right]^n.
\end{align*}
For $y\in \Omega$, as above we obtain, 
\begin{align*}
\mathrm{vol}_f(B_{r-\diam\Omega}(y)) &\le \left(\frac{n}{2a}\right)^n   \int_\Omega 
\left[ \frac{2a}{n}e^{-\frac{2ar}{n}} +h^{\frac{1}{n-1}}\right]^ne^{2ar}
e^{-f}d\mathrm{vol}.
\end{align*}
Accordingly, in order to implement the same technique as in the case of $f$ bounded, one would need to assume that \begin{equation}\label{ExpGrowthvolh}
\liminf_{r\to\infty}
e^{-2a r}
\mathrm{vol}_f(B_{r}(y)) \ge \beta >0.
\end{equation}
However, by the weighted volume comparison result \cite[Theorem 1.2 a)]{WW}, we have that under these assumptions, for any fixed $x\in M$ and any $r>0$ 
\[
\mathrm{vol}_{f}(B_{r}(x))\leq Cr^{n}e^{ar}.
\]
This is in contrast with the assumption \eqref{ExpGrowthvolh}.

\section{Elliptic Kato constant and the proof of Theorem \ref{isoKato}}\label{sec:Kato}
\begin{definition}
The Schr\"odinger operator $L\doteq-\Delta-(n-2)\mathrm{Ric}_-$ is said to be \emph{gaugeable} with gaugeability constant $\gamma$ if there exists $\varphi: M\rightarrow\mathbb{R}$ such that $1\leq\varphi\leq\gamma$ and $L\varphi=0$.
\end{definition}
The following lemma appeared as Lemma 1.14 in \cite{Ca}; see also \cite{ZZ} and \cite{CaMT} and compare with \cite[Remark 2.31]{Ca}.

\begin{lemma}\label{lem:gauge}
Let $(M^n, g)$, $n\geq3$, be a complete non-parabolic Riemannian manifold and assume that 
$(n-2)k_\infty <  1$.  Then the operator $L$ defined above is gaugeable with gaugeability constant 
\[
\gamma=\frac{1}{1-(n-2)k_\infty}.
\]
\end{lemma}
\begin{proposition}\label{prop_conformal}
Let $(M^n, g)$, $n\geq3$, be a complete non-parabolic Riemannian manifold. Let $\varphi$ be a solution of $L\varphi=0$ with $1\leq\varphi\leq \gamma$ and set $f\doteq\frac{\log\varphi}{n-2}$. Then the weighted manifold $(M^n,\widetilde{g}\doteq e^{2f}g, \varphi^{-1}\,d\mathrm{vol}_{\widetilde g}=e^{2f}\,d\mathrm{vol}_g)$ satisfies
\[
\mathrm{\widetilde{Ric}}_{\log\varphi}\geq 0
\]
\end{proposition}
Note that if $(n-2)k_\infty<1$, then  functions $\varphi$ and $f$ as in the statement exist thanks to Lemma \ref{lem:gauge}. 

\begin{proof}
Set $\widetilde{g}=e^{2f}g$ and denote by $\widetilde{\mathrm{Ric}}$ and by $\widetilde{\mathrm{Hess}}$ the Ricci tensor and the Hessian associated to $\widetilde{g}$. Standard formulas in Riemannian geometry (see e.g. \cite[1.159]{Besse}) imply that
\begin{align*}
\widetilde{\mathrm{Ric}}&=\mathrm{Ric}-(n-2)(\mathrm{Hess}\,f-df\otimes df)-(\Delta f+(n-2)|\nabla f|^2)g\\
\widetilde{\mathrm{Hess}}\,u&=\mathrm{Hess}\,u-df\otimes du+g(\nabla f,\nabla u)g-du\otimes df.
\end{align*}
Hence,
\begin{align}
  \mathrm{\widetilde{Ric}}_{\log\varphi}=\mathrm{\widetilde{Ric}}_{(n-2)f}&=\mathrm{\widetilde{Ric}}+(n-2)\widetilde{\mathrm{Hess}}\,f\nonumber\\ 
  &=\mathrm{Ric}-\Delta f g-(n-2)df\otimes df. \label{e:riccitildeh}
\end{align}
Now, by the definition of $f$ and by the fact that $\varphi$ is a solution of $L\varphi=0$ with $1\leq\varphi\leq \gamma$, we get
\begin{align*}
\Delta f&=\frac{1}{n-2}\left(\frac{\Delta \varphi}{\varphi}-\frac{|\nabla \varphi|^2}{\varphi^2}\right)\\
&=\frac{1}{n-2}\left(-(n-2)\mathrm{Ric}_--(n-2)^2|\nabla f|^2\right).
\end{align*}
Inserting the last expression in equation \eqref{e:riccitildeh} and using the Cauchy-Schwarz inequality, we obtain the desired conclusion.
\end{proof}

\begin{proof}[Proof (of Theorem \ref{isoKato})]

By Proposition \ref{prop_conformal} we know that the weighted manifold $(M^n,\widetilde{g}\doteq e^{2f}g, \varphi^{-1}\,d\mathrm{vol}_{\widetilde g}=e^{2f}\,d\mathrm{vol}_g)$ satisfies
\[
\mathrm{\widetilde{Ric}}_{\log\varphi}\geq 0
\]
where
$f\doteq\frac{\log\varphi}{n-2}$ is such that
\[
0\le f \le \frac{\log \gamma}{n-2},
\]
and $\gamma$ and $\varphi$ are given by \ref{lem:gauge}. Let us denote by $\tilde B_r(y)$ and $\widetilde {\mathrm{vol}}_{(n-2)f}$ respectively the geodesic ball centered at $y$ and the weighted volume of sets of $M$ with respect to the metric $\tilde g$. In particular 
$\widetilde{\mathrm{vol}}_{(n-2)f}(K)=\int_K e^{2f}\,d\mathrm{vol}$ for all measurable $K\subset M$. For the ease of notation, set $\lambda\doteq n-2$.
Since $1\le e^{2f} \le \gamma^{2/\lambda}$,
then $\tilde B_r(y)\supset B_{\gamma^{-1/\lambda}r}(y)$ and 
\[
\frac{\widetilde{\mathrm{vol}}_{(n-2)f}(\tilde B_{r}(y))}{r^n}\ge \frac{\mathrm{vol}(B_{\gamma^{-1/\lambda}r}(y))}{r^n} = \gamma^{-n/\lambda}\frac{\mathrm{vol}(B_{\gamma^{-1/\lambda}r}(y))}{(\gamma^{-1/\lambda} r)^n}.
\]
In particular,
\[
\limsup_{r\to\infty}\frac{\widetilde{\mathrm{vol}}_{(n-2)f}(\tilde B_{r}(y))}{r^n}\ge \tilde\beta\doteq \gamma^{-n/\lambda}\beta.
\]
An application of Theorem \ref{hbounded} gives that
\begin{equation*}
\left(\int_\Omega \ h^{\frac{n}{n-1}}e^{-(n-2)f}d\widetilde{\mathrm{vol}}\right)^{\frac{n-1}{n}}\le\frac{e^{\frac{4k}n}}{n \tilde\beta^{\frac{1}{n}}}\left(\int_\Omega |\widetilde{\nabla} h|_{\tilde g}e^{-(n-2)f}\,d\widetilde{\mathrm{vol}} + \int_{\partial \Omega} h e^{-(n-2)f}\,d\widetilde{\mathrm{vol}}_{n-1}\right).
\end{equation*}
for any smooth function $h$ on $\Omega$, where $k=(n-2)\log\gamma/(2\lambda)$ according to Remark \ref{rmk:k}. This latter can be written as
\begin{equation*}
\left(\int_\Omega \ h^{\frac{n}{n-1}}e^{2f}d{\mathrm{vol}}\right)^{\frac{n-1}{n}}\le\frac{\gamma^{\frac{3n-4}{n\lambda}}}{n\beta^{\frac{1}{n}}}\left(\int_\Omega |\nabla h|_{g}e^f\,d{\mathrm{vol}} + \int_{\partial \Omega} h e^{f}\,d{\mathrm{vol}}_{n-1}\right),
\end{equation*}
which in turn implies
\begin{equation*}
\left(\int_\Omega \ h^{\frac{n}{n-1}}d{\mathrm{vol}}\right)^{\frac{n-1}{n}}\le\frac{\gamma^{\frac{4n-4}{n\lambda}}}{n\beta^{\frac{1}{n}}}\left(\int_\Omega |\nabla h|_{g}\,d{\mathrm{vol}} + \int_{\partial \Omega} h \,d{\mathrm{vol}}_{n-1}\right).
\end{equation*}
\end{proof}

\section{Kato constant of manifolds with asymptotically non-negative Ricci curvature}\label{sec:asymp}

When $\mathrm{Ric}\geq 0$, Li and Yau proved the following estimate for the Green function, if it exists:
\begin{equation}\label{e:GreenEstimates}
G(x,y)\leq C(n)\int_{d(x,y)}^{+\infty} \frac{t}{\mathrm{vol}(B_t(y))}dt,
\end{equation}
where $C(n)$ is a purely dimensional positive constant, \cite{LY}.

Recall that $M$ is said to satisfy condition \textbf{(VC)} with respect to a pole $o\in M$ if there exists a constant $\xi>0$ such that for all $r$ and $x\in \partial B_r(o)$, $\mathrm{vol}(B_r(o))\leq \xi \mathrm{vol}(B_{r/2}(x))$.    In \cite[Theorem 1.9]{LiTam95}, Li and Tam proved that when the manifold satisfies condition \textbf{(VC)} and has negative part of the Ricci curvature which decays at least quadratically,
then the non-parabolicity is equivalent to the volume growth condition
\begin{equation}\label{e:largevolumegrowth}
\frac{t}{\mathrm{vol}(B_t(p))}\in L^1(+\infty).
\end{equation}
Moreover, in the same assumptions it is possible to obtain upper estimates \`a la Li-Yau on each non-parabolic end. Also,  they proved that the assumptions above are granted when the manifold has asymptotically non-negative sectional curvature and only one end. 
More precisely, they proved the following result (see also \cite[Theorem A.1]{Fa} where the expression of $C(K,n)$ is computed).
\begin{proposition}
    \label{lem:fan}
Let $M$ be a complete noncompact manifold of dimension $n$ with only one end,
and let $o\in M$ be a fixed point. Suppose that
$\int_1^\infty
\frac{t}
{\mathrm{vol}(B_t(p))} \,dt < \infty$. Suppose that the Ricci curvature of
$M$ satisfies $\operatorname{Ric}(y) \ge -\frac{(n-1)K}
{(1+r_o(y))^2}$ for some $K > 0$, where $r_o(y) = d(y, o)$. Suppose also that $M$ satisfies \textbf{(VC)}
for some $\xi>0$ with respect to the point $o$. Let $y\in M_r$, where $M_r$ is the union of the unbounded connected components of $M\setminus \bar B_r(o)$. Then
\[
G(o,y)\le C(K,n,\xi)\int_{r}^\infty\frac{t}{\mathrm{vol}(B_t(o))}\,dt,
\]
where $C(K,n,\xi)=K^{-n/2}\exp((1+\xi)\exp(C(n)(1+\sqrt{K}))).$
\end{proposition}

The content of the following result is that if we assume asymptotically non-negative curvature together with $\textbf{(VC)}$ and positive asymptotic volume ratio, then the volumes of all the balls in $M$ are comparable to the Euclidean ones, that is to say, $(M, g)$ is $n$-Ahlfors regular. 

%

\begin{lemma}\label{lem:VC}
Let $(M^n, g)$, $n\geq3$, be a complete Riemannian manifold satisfying
\begin{equation}\label{e:vol 4.2}
\limsup_{r\to\infty}\frac{\mathrm{vol}(B_r(q))}{r^n}\geq \beta>0
\end{equation}
for some (hence any) $q\in M$. Let $o\in M$ be a fixed pole.
Suppose that \[\operatorname{Ric}(y) \ge -\frac{(n-1)K}
{(1+d(o,y))^\alpha}\] 
for some $K > 0$, $\alpha>2$. Suppose also that $M$ satisfies the condition $\textbf{(VC)}$ with constant $\xi>0$ with respect to the reference point $o$. Then, there exists a constant $v_0>1$ depending on $n,\alpha,\beta,\xi$ and $K$ such that for every $y\in M$ and $r>0$,
\[
v_0^{-1}r^n \le \mathrm{Vol}\left(B_r(y)\right)\leq v_0 r^n.
\]
\end{lemma}

\begin{proof}
Set $\frac{K}
{(1+d(o,y))^\alpha}\doteq\lambda(d(o,y))$ and let 
$b_0=b_0(K,\alpha)\doteq\int_0^\infty t\lambda(t)<\infty$. Set $\delta_x=d(x,o)$.
A version of the Bishop-Gromov inequality ensures that 
\begin{equation}\label{e:bazan}
\frac{\operatorname{Vol}(B_R(x))}{\operatorname{Vol}(B_r(x))}\leq e^{(n-1)b_0}\left(\frac{R}{r}\right)^n
\end{equation}
for all point $x\in M$ and all $0<r<R<\delta_x$, and for all $0<r<R<\infty$ when $x=o$; see \cite[Lemma 1]{Baz-RMC}. In particular, for all $r>0$,
\[
\frac{\operatorname{Vol}(B_r(o))}{r^n}\geq \limsup_{R\to\infty}e^{-(n-1)b_0}\frac{\operatorname{Vol}(B_R(o))}{R^n} \ge e^{-(n-1)b_0}\beta.
\]
     We first prove the lower bound. If $r\le\delta_x/2$, then 
     \begin{align*}
              \mathrm{vol}(B_r(x))&\ge C(n,b_0)\left(\frac{2r}{\delta_x}\right)^n\mathrm{vol}(B_{\delta_x/2}(x))\ge C(n,b_0,\xi)\frac{\mathrm{vol}(B_{\delta_x}(o))     
     }{\delta_x^n} r^n  \ge C(n,b_0,\xi)\beta r^n,
          \end{align*}
where $R>\delta_x$ and the last inequality is obtained as $R\to\infty$. If $\delta_x/2<r\le 2\delta_x$, then 
     \[
     \mathrm{vol}(B_r(x))\ge\mathrm{vol}(B_{\delta_x/2}(x))\ge
     \xi
     \mathrm{vol}(B_{\delta_x}(o))\ge  C(n,b_0,\xi)\beta\delta_x^n\ge C(n,b_0,\xi)\beta r^n.
     \]
     If $r> 2\delta_x$, then  
     \[
     \mathrm{vol}(B_r(x))\ge\mathrm{vol}(B_{r-\delta_x}(o))\ge
C(n,b_0)\beta
    (r-\delta_x)^n\ge C(n,b_0)\beta r^n.
     \]
Concerning the upper bound, if $r\le \delta_x$, then
\[
\mathrm{vol}(B_r(x))\le C(n,b_0)\lim_{s\to 0} \frac{\mathrm{vol}(B_s(x))}{s^n} r^n = C(n,b_0) r^n,
\]
while if $r>\delta_x$,
\[
\mathrm{vol}(B_r(x))\le \mathrm{vol}(B_{2r}(o))\le C(n,b_0) r^{n}.
\]
\end{proof}

\begin{remark}\label{rmk:bazan}
\rm{In a previous version of this manuscript we stated a stronger version of Lemma \ref{lem:VC} (and thus of Theorem \ref{coro:asymp}) which did not assume condition \textbf{(VC)}. This was based on Theorem 1 of \cite{Baz-RMC}, which affirms that \eqref{e:bazan} holds for any $R>r>0$. As remarked in \cite[Section A.2]{CDM},
an inspection of the proof of \cite[Theorem 1]{Baz-RMC} shows that such a comparison is true for $0<r<R<\delta_x$, while for larger $R>r$ an additional polynomial term $\delta_x^{n-1}$ appears.
}
\end{remark}

In the assumptions of Lemma \ref{lem:VC}, also Proposition \ref{lem:fan} applies, and we obtain that for all $y\in M_r$,
\begin{equation}\label{e:GreenFan}
    G(o,y)\le C(n,K,\alpha,\beta,\xi) d(o,y)^{2-n}.
\end{equation}
Now, we would like to use estimates for the Green function to bound the elliptic Kato constant 
\begin{equation}\label{e:kinfty_sec4}
k_\infty\doteq k_\infty(\mathrm{Ric}_-)=\sup_{x\in M}\int_M G(x,y)\mathrm{Ric}_-(y)\,d\mathrm{vol}(y).
\end{equation}
The problem is that the machinery introduced by Li-Tam gives estimates of the Green function $G(o, \cdot)$ centered at the pole $o$ of $M$, i.e. the point with respect to which the curvature asymptotic condition is assumed. It is straightforward to show that if
$\operatorname{Ric}(y) \ge -(n-1)K(1+d(o,y))^{-2}$, then
\begin{align}\label{e:K tilde}
\operatorname{Ric}(y) \ge  -\frac{(n-1)K}
{(1+d(y,p))^2}\frac{(1+d(y,p))^2}
{(1+d(y,o))^2}
\ge -\frac{(n-1)\tilde K}
{(1+d(y,p))^2}
\end{align}
with $\tilde K=2K\left(1+d(o,p)\right)^2$.
In particular, assuming also \eqref{e:vol 4.2}, the integral in
\eqref{e:kinfty_sec4} converges \footnote{up to deal with points in the bounded connected components of $M\setminus \bar B_r(p)$; see the proof of Theorem \ref{th:green estimate} in Section \ref{Sect:GreenEstimates}.} 
and the elliptic Kato constant is well defined, although possibly infinite. However, the dependence on $K$ of the constant in \eqref{e:GreenFan} is doubly exponential (see \cite[Theorem A.1]{Fa}), so that to conclude the smallness of $k_\infty$, apparently, a very fast doubly exponential decay of the curvature is needed. 

To overcome such an unsatisfactory requirement, in this paper we follow a different approach. Namely, we first recall that the curvature assumption
\begin{equation}\label{e:asymp bound sec4}\operatorname{Ric}(y) \ge -\frac{(n-1)K}{1+d(o,y)^{2}}
\end{equation} 
also imply that 
\begin{align}\label{e:asymp translated}
\operatorname{Ric}(y) \ge  -\frac{(n-1)K}
{1+(d(y,p))-d(o,p))^2}
\end{align}
for all $y\in M$, with $o$ the pole of the manifold, and $p$ the point with respect to which we want to estimate the Green function.
One of the main ideas of Li-Tam's proof is roughly the following. The curvature bound \eqref{e:asymp bound sec4} implies that any annular neighborhood about $\partial B_R(o)$ of width $\asymp R$, can be covered by a uniform (with respect to $R$) number of balls of radius $\rho\doteq R/4$. Denoting by $\theta$ the infimum of the curvature on one of these balls, one has that $\theta \cdot \rho^2$ is uniformly bounded. Accordingly, Harnack and mean values inequalities hold on such balls, with constants independent of $R$. A tricky covering argument thus permits to estimate $G(x,o)$ in terms of $G(y,o)$ for any couple of point $x,y\in M_r$. On non-parabolic ends one can let $y$ escape to infinity in such a way that $G(y,o)\to 0$. This permits to deduce an effective upper bound on $G(x,o)$.  
Suppose that we have instead the assumption \eqref{e:asymp translated}, with $p$ playing the role of the new reference point (in order to deduce an estimate for $G(p,x)$), and $d(o,p)$ interpreted as a parameter. Outside the annulus $A\doteq B_{6 d_{(o,p)}}(p)\setminus B_{d_{(o,p)}/6}(p)$, the curvature bound \eqref{e:asymp translated} implies a bound of the same form of \eqref{e:asymp bound sec4} up to a multiplicative numerical constant, i.e.
\begin{equation*}\operatorname{Ric}(y) \ge -\frac{C(n-1)K}{1+d(p,y)^{2}}.
\end{equation*} 
Accordingly, in this region one can recover the same local estimates as in Li-Tam. Conversely, in the annulus $A$ one can only assume a constant lower bound on the curvature, hence a constant value of $\theta$ independent of $r$. Thus, the Harnack and mean value inequalities hold on ball of constant radius. Obviously, the number of balls which are necessary to cover $A$ increases with the width of the annulus, i.e. with $d(o,p)$. However, the dependence is polynomial and not doubly exponential. The proof of the following theorem is rather technical and involved, and it is thus postponed to Section \ref{Sect:GreenEstimates}.

\begin{theorem}\label{th:green estimate}
Let $(M^n, g)$, $n\geq3$, be a complete Riemannian manifold satisfying
\[
\limsup_{r\to\infty}\frac{\mathrm{vol}(B_r(q))}{r^n}\geq \beta>0
\]
for some (hence any) $q\in M$.
Let $\alpha>2$ and assume that for all $x\in M$
\[
\mathrm{Ric}(x)\geq -\frac{(n-1)K}{1+d(o,x)^\alpha},\] 
where $o\in M$ is some fixed reference point. Suppose also that $M$ satisfies the condition $\textbf{(VC)}$ with constant $\xi>0$ with respect to the reference point $o$. Fix $p\in M$ and let $x\in \partial B_{r}(p)$ for some $r\leq 1$. Then
\begin{equation}\label{e:Green}
    \ G(x,p)\leq C(n,\alpha,K,\beta,\xi)\max\left\{1,d(o,p)^{3n-2}\right\}r^{2-n}.
\end{equation}
\end{theorem}

\begin{remark}\label{rmk:green estimates}
Note that condition \textbf{(VC)} is authomatically satisfied when $M$ has only one end and asymptotically non-negative sectional curvature, i.e. $\mathrm{Sect}_x\ge \lambda(d(x,o))$, with $b_{0}\doteq\int_0^\infty s\lambda(s)\,ds<\infty$; see \cite[Proposition 3.8]{LiTam95}. However, under this stronger set of assumptions, the estimate \eqref{e:Green} can be improved. Indeed, as a consequence of \cite[Corollary 7.14, Lemma 2.7]{GSC} one can actually prove that for any $p,x\in M$ it holds that
\[
\ \ G(x,p)\leq C(n,\alpha,K,\beta,b_{0})d(x,p)^{2-n}.
\]
We are grateful to the anonymous referee for bringing the reference \cite{GSC} to our attention.

For later purposes, let us also remark that the conclusion of Lemma \ref{lem:VC} holds true also in this setting without further assumptions on the Ricci curvature, as it is clear from the proof of that lemma.
\end{remark}

In the last part of this section, we are going to use Theorem \ref{th:green estimate} to estimate $k_\infty$.
\medskip

\begin{proof}[Proof of Theorem \ref{coro:asymp}] 

Assume first that we are in the assumptions (a). From Theorem \ref{th:green estimate}, we know that  
\[
G(x,p)\leq C(n,\alpha,K,\beta,\xi) \max\left\{1,d(o,p)^{3n-2}\right\}d(x,p)^{2-n}.\]
Set \[B_{1,p}= \left\{x\in M\ :\ d(p,x)\le \frac{1}{2}d(o,p)\right\},\qquad B_{2,p}= \left\{x\in M\ :\ d(p,x)\le 2d(o,p)\right\}.\]
Then 
\[
k_\infty\le (n-1)\,K\,C \sup_{p\in M}\left[ \int_{B_{1,p}} \mathcal F(x)\,d\mathrm{vol}(x) + \int_{B_{2,p}\setminus B_{1,p}} \mathcal F(x)\,d\mathrm{vol}(x) +\int_{M\setminus B_{2,p}} \mathcal F(x)\,d\mathrm{vol}(x) 
\right],\]
where
\[\mathcal F(x)=\frac{\max\left\{1,d(o,p)^{3n-2}\right\}}{d(x,p)^{n-2}}\frac{1}{1+d(o,x)^{\alpha}}.\]
We are going to estimate the three integral above separately.
First, since $d(x,o)\ge d(p,o)/2$ on $B_{1,p}$, we have
\[
\int_{B_{1,p}} \mathcal F(x)\,d\mathrm{vol}(x) \le 
\frac{\max\left\{1,d(o,p)^{3n-2}\right\}}{1+(d(o,p)/2)^{\alpha}}\int_{B_{1,p}} \frac{1}{d(x,p)^{n-2}}
\,d\mathrm{vol}(x). 
\]
Integrating by part and using Lemma \ref{lem:VC}, we compute
\begin{align*}
\int_{B_{1,p}} \frac{1}{d(x,p)^{n-2}}
\,d\mathrm{vol}(x) &= \int_0^{d(o,p)/2} A(r)r^{2-n}\,dr \\
&=   V\left(\frac{d(o,p)}{2}\right)\,\left(\frac{d(o,p)}{2}\right)^{2-n} + (n-2) \int_0^{d(o,p)/2} {V}(r)r^{1-n}\,dr\\
&\le \frac{n}{8}v_0\,d(o,p)^2,
\end{align*}
where $V(r)$ and $A(r)$ denotes the volume $B_r(p)$ and the area of its boundary, respectively. Hence,
\[
\int_{B_{1,p}} \mathcal F(x)\,d\mathrm{vol}(x) \le 
\frac{\max\left\{1,d(o,p)^{3n-2}\right\}}{1+(d(o,p)/2)^{\alpha}}\frac{n}{8}v_0\,d(o,p)^2 \le C(n,v_0)
\]
as soon as $\alpha \ge 3n$.
Similarly, as $d(x,p)\ge d(o,p)/2$ on $B_{2,p}\setminus B_{1,p}$, and this latter is contained in $B_{3 d(o,p)}(o)$, we get
\[
\int_{B_{2,p}\setminus B_{1,p}} \mathcal F(x)\,d\mathrm{vol}(x) \le 2^{n-2}\frac{\max\left\{1,d(o,p)^{3n-2}\right\}}{d(o,p)^{n-2}}\int_{B_{3d(o,p)}(o)} \frac{1}{1+d(o,x)^{\alpha}}
\,d\mathrm{vol}(x)\le C(n,v_0)
\]
as soon as $\alpha \ge 3n$, because
\[
\int_{B_{3d(o,p)}(o)} \frac{1}{1+d(o,x)^{\alpha}}d\mathrm{vol}(x)\le C(n,v_0)\begin{cases}
d(o,p)^n&\text{if } d(o,p)\le 1/3,\\
d(o,p)^{n-\alpha}&\text{if } d(o,p)>1/3.    
\end{cases}
\]
Finally, as $d(x,o)\ge d(x,p)/2$ and $d(o,p)\le d(p,x)/2$ on $M\setminus B_{2,p}$, and this latter is nothing but $M\setminus B_{2 d(o,p)}(p)$, we get
\begin{align*}
    \int_{M\setminus B_{2,p}} \mathcal F(x)\,d\mathrm{vol}(x) &\le 
\max\left\{1,d(o,p)^{3n-2}\right\}\int_{M\setminus B_{2d(o,p)}(p)} \frac{2^\alpha}{d(p,x)^{n-2}(1+d(p,x)^{\alpha})}
\,d\mathrm{vol}(x)\\
& =
2^\alpha\int_{2d(o,p)}^{\infty} \frac{\max\left\{1,r^{3n-2}\right\}\, A(r)}{r^{n-2}(1+r^{\alpha})}\,dr\\
&\le 2^\alpha\int_{0}^{1} \frac{A(r)}{r^{n-2}(1+r^{\alpha})}\,dr+2^\alpha\int_{1}^{\infty} \frac{r^{3n-2}\, A(r)}{r^{n-2}(1+r^{\alpha})}\,dr.
\end{align*}
Integrating by part as in the first case, we finally get that the right hand side above has an upper bound which depends only on $\alpha$, $n$ and $v_0$, as soon as $\alpha > 3n$. This concludes the proof in the assumptions (a).

Suppose now that we are in assumptions (b). We observe preliminarily that in this case one has
 \[
\sup_{r>0} r^2\lambda(r) \le 2b_0\quad\text{and}\quad \mathrm{Ric}(x)\geq -\frac{(n-1)2b_0}{d(o,x)^2},\]
see \cite[Lemma 1.1]{Abr}. Moreover, $G(x,p)\leq C(n,\alpha,K,\beta,b_{0}) d(x,p)^{2-n}$ by Remark \ref{rmk:green estimates}. We have thus to estimate 
\[
 \int_{B_{1,p}} \mathcal F(x)\,d\mathrm{vol}(x) + \int_{B_{2,p}\setminus B_{1,p}} \mathcal F(x)\,d\mathrm{vol}(x) +\int_{M\setminus B_{2,p}} \frac{\lambda(d(o,x))}{d(x,p)^{n-2}}\,d\mathrm{vol}(x) 
,\]
where now
\[\mathcal F(x)=\frac{1}{d(x,p)^{n-2}}\frac{1}{d(o,x)^{2}}.\]
Computing as above the first two integrals can be upper bounded by a constant which depends only on $n$ and $v_0$, where $v_0$ is the constant appearing in Lemma \ref{lem:VC}; see also Remark \ref{rmk:green estimates}. Concerning the last integral, recall that $d(x,o)\ge d(x,p)/2$ and $d(o,p)\le d(p,x)/2$ on $M\setminus B_{2,p}$. Then, using that $\lambda$ and $r\mapsto \lambda(r/2)r^{2-n}$ are nonincreasing, we have 
\begin{align*}
    \int_{M\setminus B_{2,p}} \frac{\lambda(d(o,x))}{d(x,p)^{n-2}}\,d\mathrm{vol}(x) 
    &\le  \int_{M\setminus B_{2,p}}
\frac{\lambda(d(p,x)/2)}{d(x,p)^{n-2}}\,d\mathrm{vol}(x)\\
    &\le \int_{0}^\infty \frac{\lambda(r/2)\,A(r)}{r^{n-2}}\,dr\\
    &\le \lim_{r\to\infty}\frac{\lambda(r/2)\,V(r)}{r^{n-2}} - \int_0^\infty V(r)\left(\frac{\lambda(r/2)}{r^{n-2}}\right)'\,dr\\
    &\le 8\,b_0\,v_0 - v_0 \int_0^\infty r^n\left(\frac{\lambda(r/2)}{r^{n-2}}\right)'\,dr \\
    &\le 8\,b_0\,v_0 + v_0 \int_0^\infty nr^{n-1}\frac{\lambda(r/2)}{r^{n-2}}\,dr\\
    &\le 8\,b_0\,v_0 + 8\,v_0\,n\,b_0.
\end{align*}

\end{proof}

\begin{remark}
Given a measurable function $h$ on $M$, consider the function $\phi_h(x):M\to [0,+\infty)\cup\{+\infty\}$ defined by
\[
    \phi_h(x)\doteq\int_M G(x,y)h(y)\,d\mathrm{vol}(y),
\]
and note that when $h$ is regular enough (e.g. $h\in C^\infty_c(M)$) then $\phi_h$ solves
\begin{equation}\label{e:poisson pb}
    - \Delta \phi_h (x)=  h(x).
\end{equation}
The argument above proves that in the assumptions of Theorem \ref{th:green estimate} the Poisson problem \eqref{e:poisson pb} has a bounded solution whenever $h\in C^\infty(M)$
    satisfies $
|h(x)|=o(d(o,x)^{-\alpha})$ for $\alpha>3n$.
Similarly, if $M$ has only one end and asymptotically nonnegative curvature, then \eqref{e:poisson pb} has a bounded solution if $|h(x)|\le \lambda(d(x,o))$ where $\lambda$ is a non-increasing function satisfying \eqref{e:asymp curv}.
The search for conditions ensuring the existence and the boundedness of solutions to the Poisson problem on complete non-compact manifolds is a quite active field (see for instance \cite{MSW,CMP,Fa} and references therein). As a byproduct, our work gives also a contribution in that direction.
\end{remark}

\section{Estimates on the Green function on manifolds with asymptotically non-negative Ricci curvature}\label{Sect:GreenEstimates}

In the present section we 
prove the following result, from which Theorem \ref{th:green estimate} can be easily deduced.

\begin{theorem}\label{th:green estimate_abstract}
Let $(M^n, g)$, $n\geq3$, be a complete Riemannian manifold satisfying
\begin{equation}\label{e:euclvol}
\ v_{0} r^n\leq \mathrm{vol}(B_{r}(x))\leq V_{0}r^{n}\,\qquad \forall\,x\in M,\,\forall r>0.
\end{equation}
Assume that \begin{equation}\label{e:Ric delta}
\mathrm{Ric}\geq -\frac{(n-1)K}{1+(d(p,x)-\delta)^2},\qquad\forall\, x\in M\end{equation} for some fixed reference point $p\in M$ and some $\delta>0$. Let $x\in M\setminus B_{r}(p)$ for some $r\leq 1$. Then
\begin{equation}
\label{e:Green abstract}   
\ G(p,x)\leq C(n,K,v_0,V_0)\max\left\{1,\delta^{3n-2}\right\}r^{2-n}.
\end{equation}
\end{theorem}
\begin{remark}
Clearly, in the statements of theorems \ref{th:green estimate_abstract} and \ref{th:green estimate} the value $1$ of the upper bound required on $r$  does not play a special role. Thus, these theorems remain valid for $r\le r_0$ for any chosen $r_0>0$. However, the resulting constants in \eqref{e:Green abstract} and \eqref{e:Green} will depend exponentially on $r_0$. 
Instead, one can prove theorems \ref{th:green estimate_abstract} and \ref{th:green estimate} for any $r>0$ with a constant independent of $r$ under the additional assumption that
\begin{equation}\label{e:bounded connected comp}
\sup \diam \mathcal M^{bdd}_{r,j}    \le C\,r
\end{equation}
for some $C>0$ independent of $R$, where the $\sup$ is taken over all the bounded connected component $\mathcal M^{bdd}_{r,j}$ of $M\setminus \overline{B_r(p)}$. Note that condition \eqref{e:bounded connected comp} is implied by the (in principle stronger) property that $M$ has \textit{relatively connected annuli} with respect to $p$. This property, usually denoted as $\mathrm{(RCA)}$ in the literature, says that there exists a constant $C_A> 1$ such that for any $r>C_A^2$ and all $x,y\in M$ such that $d(p,x)=d(p,y)=r$, there exists a continuous path $\gamma:[0,1]\to M$ with $\gamma(0)=x$, $\gamma(1)=y$, whose image is contained in $B_{C_Ar}(p)\setminus B_{r/C_A}(p)$; see for instance \cite[Definition 5.1]{GSC}. Under this assumption, necessarly $\diam \mathcal M^{bdd}_{r,j}    \le C_A^2\, r$. Nice conditions ensuring the validity of $\mathrm{(RCA)}$ have been proposed in \cite[Proposition 0.4]{Min}. In particular, under the assumption \eqref{e:euclvol}, $\mathrm{(RCA)}$ is implied by the validity of a form of the scale-invariant $L^p$-Poincaré inequality.
\end{remark}

\begin{proof}[Proof (of Theorem \ref{th:green estimate})]

By Lemma \ref{lem:VC}, $M$ enjoys property \eqref{e:euclvol}.  
Moreover, as a consequence of the triangular inequality,
\[
\mathrm{Ric}(x)\geq -\frac{(n-1)K}{1+d(o,x)^\alpha}\geq -\frac{(n-1)K}{1+d(o,x)^2}
\geq -\frac{(n-1)K}{1+(d(p,x)-\delta)^2},
\]
with $\delta =d(o,p)$. Then, an application of Theorem \ref{th:green estimate_abstract} permits to conclude.
\end{proof}

The remaining part of this section will be devoted to prove Theorem \ref{th:green estimate_abstract}. In the following $p\in M$ is a fixed point and for $r>0$ we will denote by $M_{r}$ the union of the unbounded components of $M\setminus \overline{B_{r}(p)}$.
To simplify the notation we will omit the measure element $d\mathrm{vol}$ in the integrals.
\medskip

Let us first suppose that $\delta=d(o,p) \ge 1$. 

We begin with the following lemma, which improves on  \cite[Lemma 1.4]{LiTam95} under the assumption \eqref{e:euclvol}.
\begin{lemma}
 Suppose that $M$ satisfies condition \eqref{e:euclvol}. Then, for all $R>0$ and $0<\alpha\leq (Q-1)/4$, $B_{QR}(p)\setminus \overline{B_{R}(p)}$ can be covered by $l$ geodesic balls of radii $\alpha R$ with centers in $B_{QR}(p)\setminus \overline{B_{R}(p)}$ where
\[
\ l\leq C(v_{0}, V_{0}, \alpha, Q, n)\doteq v_{0}^{-1}\left(\frac{2}{\alpha}\right)^{n}\left[V_{0}\left(Q+\frac{\alpha}{2}\right)^n-v_{0}\left(1-\frac{\alpha}{2}\right)^{n}\right].
\]
\end{lemma}

\begin{proof}
Let $l$ be the maximal number of disjoint geodesic balls of radii $\frac{\alpha R}{2}$ with centers $x_1,\dots,x_l$ in $B_{QR}(p)\setminus \overline{B_{R}(p)}$. In particular, $B_{QR}(p)\setminus \overline{B_{R}(p)} \subset \cup_{j=1}^l B_{\alpha r} (x_j)$. In order to bound $l$, let us compute 
\begin{align*}
l\,v_{0}\left(\frac{\alpha R}{2}\right)^{n}\leq&\mathrm{vol}(\bigcup_{j=1}^{l}B_{\frac{\alpha R}{2}}(x_{j}))\\ \leq& \mathrm{vol}(B_{QR+\frac{\alpha R}{2}}(p))-\mathrm{vol}(B_{R-\frac{\alpha R}{2}}(p))\\ \leq& V_{0}\left(Q+\frac{\alpha}{2}\right)^{n}R^{n}-v_{0}\left(1-\frac{\alpha}{2}\right)^{n}R^{n}.
\end{align*}
Hence
\[
\ l\leq v_{0}^{-1}\left(\frac{2}{\alpha}\right)^{n}\left[V_{0}(Q+\frac{\alpha}{2})^{n}-v_{0}(1-\frac{\alpha}{2})^{n}\right]\doteq C(v_{0}, V_{0}, \alpha, Q, n).
\]
\end{proof}

Choosing respectively $Q=2$, $\alpha=\frac{1}{4}$ and $Q=2^{6}$, $\alpha=\frac{1}{12R}$ we immediately obtain the following

\begin{corollary}\label{coro:covering}
Suppose that $M$ satisfies condition \eqref{e:euclvol}. Then,
\begin{itemize}
\item[(a)] $B_{2R}(p)\setminus \overline{B_{R}(p)}$ can be covered by $\ell=\ell(n,v_{0}, V_{0})$ balls of radius $\frac{R}{4}$ with centers in $B_{2R}\setminus \overline{B_{R}}$.
\item[(b)] For $R>\frac{1}{6}$, $B_{2^{6}R}(p)\setminus \overline {B_{R}}(p)$ can be covered by $\hat{\ell}=\hat\ell(n, v_{0}, V_{0})R^{n}$ balls of radius $\frac{\rho}{4}\doteq\frac{1}{48}$.
\end{itemize}
\end{corollary}

Let us recall the following 
\begin{lemma}[Corollary 14.8 in \cite{Li_book} with $q=1$]\label{lemma:li} Assume that $\mathrm{Ric}\geq -(n-1)\theta$ on $B_{4R}(x_{0})$ for some $\theta\ge 0$. Let $0<\sigma<\frac{1}{2}$, $\lambda\geq 0$. Then there exists $C=C(\sigma, \lambda R^2, n, \sqrt{\theta} R)$ such that for all $g\geq 0$ such that
\[
\ \Delta g \geq -\lambda g \qquad \mathrm{in}\,B_{2R}(x_{0})
\]
we have
\[
g(x_{0})\leq \sup_{B_{(1-\sigma)R}(x_{0})}g\leq C\mathrm{vol}(B_{R}(x_{0}))^{-1}\int_{B_{R}(x_{0})}g.
\]
\end{lemma}

By the proof of the lemma (see Corollary 14.8 in \cite{Li_book}) one can quantify the constant as
\begin{align*}
C=&C(n)
\frac{\mathrm{vol}(\mathbb{B}_{2R}^{-\theta})}{R^{n+2}}\frac{e^{-\frac{3\lambda}{32}R^2}-e^{- \frac{\lambda}{4}R^{2}}}{\lambda}(R\sqrt{\theta}+1)e^{C(n)R\sqrt{\theta}}e^{\frac{\lambda R^2}{4}}\\
=&C(n)\frac{\mathrm{vol}(\mathbb{B}_{2R}^{-\theta})}{R^{n}}e^{\frac{5\lambda}{32}R^{2}}\frac{1-e^{-\frac{5\lambda R^{2}}{32}}}{\lambda R^{2}}\left(R\sqrt{\theta}+1\right)e^{C(n)R\sqrt{\theta}}\\
\leq& C(n)e^{C(n)R\sqrt{\theta}+C(n)R^{2}\lambda}.
\end{align*}

\begin{corollary}\label{coro:li ours}
 Suppose that $M$ satisfies assumptions \eqref{e:euclvol} and \eqref{e:Ric delta}. Then  
\begin{itemize}
\item[(a)] If either $s\leq\frac{\delta}{6}$ or $s\geq 6 \delta$, and $x_{0}\in B_{\frac{9}{4}s}(p)\setminus B_{\frac{3}{4}s}(p)$, then
\[
 |\nabla f|^2(x_{0})\leq C(n,K)\,\mathrm{vol}(B_{s/16}(x_{0}))^{-1}\int_{B_{\frac{s}{16}}(x_{0})}|\nabla f|^2,
\]
for any function $f$ harmonic in $B_{s/8}(x_{0})$.
\item[(b)] If $x_{0}\in B_{2^7\frac{\delta}{6}+\frac{\rho}{4}}(p)\setminus B_{\frac{\delta}{6}-\frac{\rho}{4}}(p)$, $\rho=\frac{1}{12}$, then
\[
\ |\nabla f|^2(x_{0})\leq C(n,K)\,\mathrm{vol}(B_{\rho/16}(x_{0}))^{-1}\int_{B_{\frac{\rho}{16}}(x_{0})}|\nabla f|^2,
\]
for any function $f$ harmonic in $B_{\rho/8}(x_{0})$.
\end{itemize}
\end{corollary}

\begin{proof}
\begin{itemize}
    \item[(a)] Assume either $s\leq\frac{\delta}{6}$ or $s\geq 6 \delta$ and $x_{0}\in B_{\frac{9}{4}s}(p)\setminus B_{\frac{3}{4}s}(p)$. 
On $B_{\frac{s}{4}}(x_0)$, 
\begin{align}\label{e:Ric theta}
    \mathrm{Ric}&\geq -(n-1)\sup \left\{\frac{K}{1+\left(t-\delta\right)^{2}}:\,t\in \left(\frac{s}{2},\frac{5}{2}s\right)\right\}\\
    &\ge \begin{cases} -(n-1)K\left(\frac{s}{3}\right)^{-2}&\text{if }s\geq 6\delta\\
-(n-1)K\left(\frac{7s}{2}\right)^{-2}&\text{if }s\leq \frac\delta 6
\end{cases}\nonumber\\
&\ge -\frac{(n-1)9 K}{s^2}\doteq-(n-1)\theta_a.\nonumber
\end{align}
Choose $R=\frac{s}{16}$ and $\lambda_a=2(n-1)\theta_a$.
By Bochner formula, on $B_{2R}(x_{0})$
\[
\ \Delta |\nabla f|^{2}\geq -\lambda_a|\nabla f|^{2}.
\]
Hence Lemma \eqref{lemma:li} applies to $g=|\nabla f|$ and yields (a).

\item[(b)] Now assume that $x_{0}\in B_{2^7\frac{\delta}{6}+\frac{\rho}{4}}(p)\setminus B_{\frac{\delta}{6}-\frac{\rho}{4}}(p)$, $\rho=\frac{1}{12}$.
On the whole $M$, hence in particular on $B_{\frac{\rho}{4}}(x_0)$, 
\begin{align*}
    \mathrm{Ric}&\geq -(n-1)K\doteq-(n-1)\theta_b.
\end{align*}
Choose $R=\frac{\rho}{16}$ and $\lambda_b=2(n-1)\theta_b$.
By Bochner formula, on $B_{2R}(x_{0})$
\[
\ \Delta |\nabla f|^{2}\geq -\lambda_b|\nabla f|^{2}.
\]
Hence Lemma \eqref{lemma:li} applies to $g=|\nabla f|$ and yields (b).
\end{itemize}
\end{proof}

\begin{lemma}
    \label{corollary_osc}  Suppose that $M$ satisfies assumptions \eqref{e:euclvol} and \eqref{e:Ric delta}. Then
\begin{itemize}
\item[(a)] Suppose either $s\leq\frac{\delta}{6}$ or $s\geq 6 \delta$.
Let $f$ be harmonic on $M\setminus B_{\frac{5}{8}s}(p)$ and $x_{0}\in B_{2s}(p)\setminus B_{s}(p)$. Then
\[
 \mathrm{osc}_{B_{\frac{s}{4}}(x_{0})}f\leq C(n, v_{0}, K )s^{1-n/2}\sqrt{\int_{B_{\frac{5}{16}s}(x_0)}|\nabla f|^2}.
\]
\item[(b)] If $x_{0}\in B_{2^7\frac{\delta}{6}}(p)\setminus B_{\frac{\delta}{6}}(p)$, $\rho=\frac{1}{12}$. Let $f$ be harmonic on $M\setminus B_{\frac{\delta}{6}-\frac{3}{8}\rho}(p)$. Then
\[
 \mathrm{osc}_{B_{\frac{\rho}{4}}(x_{0})}f\leq C(n, v_{0}, K )\rho^{1-n/2}\sqrt{\int_{B_{\frac{5}{16}\rho}(x_0)}|\nabla f|^2}.
\]
\end{itemize}
\end{lemma}

\begin{proof}
\begin{itemize}
\item[(a)] Suppose either $s\leq\frac{\delta}{6}$ or $s\geq 6 \delta$, and let $x_{0}\in B_{2s}(p)\setminus B_{s}(p)$. For any $y\in B_{\frac{s}{4}}(x_0)$, let $\gamma(t)$ be a minimizing geodesic from $x_0$ to $y$ with $\mathrm{length}(\gamma)=l\leq\frac{s}{4}$. Then for all $t\in [0,l]$, $\gamma(t)\in B_{\frac{9}{4}s}(p)\setminus B_{\frac{3}{4}s}(p)$ and
\[
\ |\nabla f|^{2}(\gamma(t))\leq C(n, v_{0},K)\frac{1}{s^{n}}\int_{B_{\frac{5}{16}s}(x_0)}|\nabla f|^{2},
\]
as a consequence of Corollary \ref{coro:li ours} (a). Hence
\begin{align*}
|f(x_0)-f(y)|&\leq \int_{0}^{l}|\nabla f|(\gamma(t))dt\leq \frac{s}{4}\sup_{t\in[0,l]}|\nabla f|(\gamma(t))\\
&\leq \frac{s}{4}\sqrt{C(n,v_{0},K)\frac{1}{s^{n}}\int_{B_{\frac{5}{16}s}(x_0)}|\nabla f|^{2}}.
\end{align*}
\item[(b)] Let $x_{0}\in B_{2^7\frac{\delta}{6}}(p)\setminus B_{\frac{\delta}{6}}(p)$, $\rho=\frac{1}{12}$. For any $y\in B_{\frac{\rho}{4}}(x_0)$, let $\gamma(t)$ be a minimizing geodesic from $x_0$ to $y$ with $\mathrm{length}(\gamma)=l\leq\frac{\rho}{4}$. Then for all $t\in [0,l]$, $\gamma(t)\in B_{2^7\frac{\delta}{6}+\frac{\rho}{4}}(p)\setminus B_{\frac{\delta}{6}-\frac{\rho}{4}}(p)$. Hence
\begin{align*}
|\nabla f|^{2}(\gamma(t))&\leq C(n, v_{0}, K)\frac{1}{\rho^{n}}\int_{B_{\frac{\rho}{16}}(\gamma(t))}|\nabla f|^{2}\\
&\leq C(n, v_{0}, K)\frac{1}{\rho^{n}}\int_{B_{\frac{5}{16}\rho}(x_0)}|\nabla f|^{2},
\end{align*}
as a consequence of Corollary \ref{coro:li ours} (b). Hence
\begin{align*}
|f(x_0)-f(y)|\leq&\int_{0}^{s}|\nabla f(\gamma(t))|dt\leq\frac{\rho}{4}\sup_{t\in[0,l]}|\nabla f|(\gamma(t))\\
\leq&\frac{\rho}{4}\sqrt{C(n, v_{0},k)\frac{1}{\rho^{n}}\int_{B_{\frac{5}{16}\rho}(x_0)}|\nabla f|^{2}}.
\end{align*}
\end{itemize}
\end{proof}

\begin{lemma}[Lemma 2.1 in \cite{LiTam95}]
Let $X$ be a topological space and let $f:X\to\mathbb{R}$ be a continuous function. Let $\gamma:[0,1]\to X$ be a continuous curve which is covered by finitely many non-empty open sets $U_{1},\dots,U_k$. Then
\[
\ \left|f(\gamma(0))-f(\gamma(1))\right|\leq \sum_{i=1}^{k}\mathrm{osc}_{U_{i}}f.
\]
\end{lemma}

\begin{lemma}\label{LemmaA.VIII}
Suppose that $M$ satisfies assumptions \eqref{e:euclvol} and \eqref{e:Ric delta}. Let $f$ be harmonic on $M_{r/2}$. Suppose either $r\leq \frac{\delta}{6}$ or $r\geq 6\delta$. For every $x,y$ in the same connected component of $M_{r}$, it holds
\[
\ |f(x)-f(y)|\leq C(n,K,v_{0},V_{0})\delta^{\frac{3n}{2}-1}r^{1-\frac{n}{2}}\left(\int_{M_{r/2}}|\nabla f|^{2}\right)^{\frac{1}{2}}.
\]
\end{lemma}
\begin{proof}
    Consider $x,y$ in the same connected component of $M_{r}$. Then, there exists a continuous curve $\gamma:[0,1]\to M_{r}$ with  $\gamma(0)=x$, $\gamma(1)=y$. Let $J$ be such that $\gamma\subset M_{r}\cap B_{2^{J}r}(p)=\bigcup_{j=1}^{J}A_{j}$,
where $A_{j}=M_r\cap(B_{2^{j}r}(p)\setminus B_{2^{j-1}r}(p))$. Let $j^{\prime}=\sup\left\{j\in\mathbb Z\,:\,2^{j-1}r\le\frac{\delta}{6}\right\}$, $j^{\prime\prime}=\inf\left\{j\in\mathbb N_{\ge 1}\,:\,2^{j-1}r\ge 6\delta\right\}$. 
If $r\ge 6\delta$, then $j''=1$ and
\begin{equation}
\label{e:diadic0}    
 \gamma\subseteq \bigcup_{j=j^{\prime\prime}}^{J}A_{j}.
\end{equation}
Otherwise,
\[
\ j^{\prime}> \log_{2}\frac{\delta}{6r},\quad j^{\prime\prime}< 2+\log_{2}\frac{6 \delta}{r},
\]
so that $j^{\prime\prime}-j^{\prime} < 2+\log_{2}36$, i.e. $j^{\prime\prime}\le j^{\prime}+ 7$,
and
\begin{equation}
\label{e:diadic}    
 \gamma\subseteq \bigcup_{j=1}^{j^{\prime}}A_{j}\cup \bigcup_{j=j^{\prime\prime}}^{J}A_{j}\cup\hat{A},
\end{equation}
where $\hat{A}\doteq M_r\cap(B_{2^{j^{\prime}+6}r}(p)\setminus B_{2^{j^{\prime}}r}(p))$, with the understanding that $\bigcup_{j=1}^{j^{\prime}}A_{j}$ may be empty if 
$j^\prime\le 0$.

For $j\leq j^{\prime}$ or $j\geq j^{\prime\prime}$, by Corollary \ref{coro:covering}(a) applied with $R=2^{j-1}r$, $A_{j}$ can be covered by $\ell=\ell(n, v_{0}, V_{0})$ balls $\left\{B_{i}^{(j)}\right\}_{i=1}^{\ell}$ of radius $\frac{2^{j-1}r}{4}=2^{j-3}r$ and Lemma \ref{corollary_osc}(a) applies with $s=2^{j-1}r$ to each of the $B_{i}^{(j)}$. Note that if  $B_{i}^{(j)}=B_{\frac{2^{j-1}r}{4}}(x_i^{(j)})$ is any of these balls then $B_{\frac{5}{16}2^{j-1}r}(x_i^{(j)})\subset M_{\frac{r}{2}}$.

\textbf{1st case: $r>6\delta$}. In this case, recalling \eqref{e:diadic0}, 
\begin{align*}
|f(x)-f(y)|\leq&\sum_{j= 1}^{J}\sum_{i=1}^{\ell}\mathrm{osc}_{B_{i}^{(j)}}f\\
\leq&\ell\,C(n,K,v_{0})\sum_{j\geq 1}\left(2^{j-1}r\right)^{1-\frac{n}{2}}\left(\int_{M_{r/2}}|\nabla f|^{2}\right)^{\frac{1}{2}}\\
\leq&C(n,K,v_{0},V_{0})r^{1-\frac{n}{2}}\left(\int_{M_{r/2}}|\nabla f|^{2}\right)^{\frac{1}{2}}\\
\leq&C(n,K,v_{0},V_{0})\delta^{\frac{3n}{2}-1}r^{1-\frac{n}{2}}\left(\int_{M_{r/2}}|\nabla f|^{2}\right)^{\frac{1}{2}},
\end{align*}
as $\delta>1$.

\textbf{2nd case: $r<\frac{\delta}{6}$}. In this case  $j^\prime\ge 1$ and all the term at the (RHS) of \eqref{e:diadic} are non trivial.
By Corollary \ref{coro:covering}(b) applied with $R=2^{j^\prime}r$, $\hat{A}$ can be covered by $\hat{\ell}=\hat \ell(n, v_{0}, V_{0})$ balls $\left\{\hat{B}_{i}\right\}_{i=1}^{\hat\ell}$ of radius $\frac{\rho}{4}=\frac{1}{48}$, with $\hat{\ell}\leq C(n,v_{0}, V_{0})(2^{j^\prime}r)^n\le C(n,v_{0}, V_{0})\delta^n$. Since $\hat A \subset B_{2^7\frac{\delta}{6}}(p) \setminus B_{\frac{\delta}{6}}(p)$, Lemma \ref{corollary_osc}(b) applies with $s=\rho$ to each of the $\hat{B}_{i}$. If  $\hat B_{i}=B_{\frac{\rho}{4}}(\hat x_i)$ is any of these balls then $B_{\frac{5}{16}\rho}(\hat x_i)\subset M_{\frac{r}{2}}$ because  
\[
\ \frac{r}{2}<\frac{\delta}{12}\leq \frac{\delta}{6}-\frac{5\rho}{16}\le 2^{j^\prime}r-\frac{5\rho}{16}\le d(\hat x_i,p)-\frac{5\rho}{16}.
\]
Since $\delta^{1-\frac{n}{2}}\le (6r)^{1-\frac{n}{2}}$, then
\begin{align}\label{osc1}
|f(x)-f(y)|&\leq\sum_{j=1,\ldots,j^{\prime}, j^{\prime\prime}, \ldots,J}\sum_{i=1}^{\ell}\mathrm{osc}_{B_{i}^{(j)}}f+\sum_{i=1}^{\hat{\ell}}\mathrm{osc}_{\hat{B}_{i}}f\\
&\leq C(n,K,v_0,V_0)  \left[\ell \sum_{j\ge 1}\left(2^{j-1}r\right)^{1-\frac{n}{2}}\left(\int_{M_{r/2}}|\nabla f|^2\right)^{\frac{1}{2}}+\hat{\ell}\left(\int_{M\setminus B_{\frac{\delta}{6}-\frac{5}{16}\rho}(p)}|\nabla f|^2\right)^{\frac{1}{2}}\right]\nonumber\\
&\leq C(n,K,v_{0}, V_{0})\left[r^{1-\frac{n}{2}}+\delta^{n}\right]\left(\int_{M_{r/2}}|\nabla f|^2\right)^{\frac{1}{2}},\nonumber\\
&=C(n, K, v_{0}, V_{0})\left[r^{1-\frac{n}{2}}+\delta^{1-\frac{n}{2}+n+\frac{n}{2}-1}\right]\left(\int_{M_{r/2}}|\nabla f|^{2}\right)^{\frac{1}{2}}\nonumber\\
&\leq C(n,k,v_{0},V_{0})\left(1+\delta^{\frac{3n}{2}-1}\right)r^{1-\frac{n}{2}}\left(\int_{M_{r/2}}|\nabla f|^{2}\right)^{\frac{1}{2}}\nonumber\\
&\leq C(n,K,v_{0}, V_{0}) \delta^{\frac{3n}{2}-1} r^{1-\frac{n}{2}}\left(\int_{M_{r/2}}|\nabla f|^{2}\right)^{\frac{1}{2}}.\nonumber
\end{align}

\end{proof}
\begin{lemma}[\cite{CY}]\label{lemma:harnack}
Suppose that $\mathrm{Ric}\geq -(n-1)\theta$ on $B_{R}(x_{0})$ and let $f>0$ be such that $\Delta f=0$ on $B_{R}(x_{0})$. Then
\begin{equation}\label{Est1}
C^{-1}\leq \frac{f(x)}{f(x_{0})}\leq C\qquad\forall\,x\in B_{R/2}(x_{0}),
\end{equation}
with $C=e^{C_n(1+\sqrt{\theta}R)}$.
\end{lemma}

\begin{lemma}\label{LemmaA.IX}
Suppose that $M$ satisfies assumptions \eqref{e:euclvol} and \eqref{e:Ric delta}.
Let either $r<\frac{\delta}{6}$ or $r>6\delta$, and let $f$ be positive and harmonic in $M_{r/4}$. Then, for every $x,y$ in the same connected component of $M_{r/2}$
\begin{align*}
f(x)\leq
C(n,v_0,V_0,K)\left(f(y)+\delta^{\frac{3}{2}n-1}r^{1-\frac{n}{2}}\left(\int_{M_{\frac{r}{2}}}|\nabla f|^{2}\right)^{\frac{1}{2}}\right),
\end{align*}
\end{lemma}
\begin{proof}
Let $x,\ y$ belong to the same connected component of $M_{\frac{r}{2}}$
The annulus $B_{r}(p)\setminus B_{\frac{r}{2}}(p)$ can be covered by $\ell=\ell(n,v_0,V_0)$ balls of radius $\frac{r}{8}$. Hence there exist $s\leq \ell$ balls $\left\{B_{i}\right\}_{i=1}^{s}$ of radius $\frac{r}{8}$ with center in $B_{r}(p)\setminus B_{\frac{r}{2}}(p)$ and points $z_{i}\in B_{i}\cap B_{i+1}$, $i=1,\dots,s-1$, so that $x\dot=z_0\in B_{1}$ and  $B_{s}\cap M_{r}\neq \emptyset$. Let  $z_s=\bar{x}$ be a point in $B_{s}\cap M_{r}$.
Let $B_i= B_{\frac{r}{8}}(x_i)$, $1\le i\le s$, be one of the balls above. 
Note that
$B_{\frac{r}{4}}(x_i)\subset B_{\frac{5}{4}r}(p)\setminus B_{\frac{r}{4}}(p)$. Accordingly, we can
compute as in \eqref{e:Ric theta} and deduce the curvature lower bound  
\begin{align*}
    \mathrm{Ric}&\geq -(n-1)\sup \left\{\frac{K}{1+\left(t-\delta\right)^{2}}:\,t\in \left(\frac{r}{4},\frac{5}{4}r\right)\right\}\\
    &\ge \begin{cases} -(n-1)K\left(\frac{r}{12}\right)^{-2}&\text{if }r\geq 6\delta\\
-(n-1)K\left(\frac{19s}{4}\right)^{-2}&\text{if }r\leq \frac\delta {6}
\end{cases}\nonumber\\
&\ge -\frac{(n-1)144 K}{ s^2}=:-(n-1)\theta_c\nonumber
\end{align*}
 on 
$B_{\frac{r}{4}}(x_i)$.
By Lemma \ref{lemma:harnack} applied with $R=\frac r4$, 
\[
f(z_{i})\leq C(n,K)^{2}f(z_{i+1}),\qquad i=0,\dots,s-1,\]
so that
$f(x)\le C(n,K)^{2\ell} f(\bar x)$.
Reasoning as above, there exists $\bar{y}\in M_{r}$ such that $f(\bar{y})\leq C(n,K)^{2\ell}f(y)$. Without loss on generality we can suppose that $\bar{x}$ and $\bar{y}$ are in the same connected component of $M_{r}$. Lemma \ref{LemmaA.VIII} gives
\begin{align*}
    f(x)&\leq C(n,K)^{2 \ell} f(\bar{x})\leq C(n,K)^{2\ell}\left(f(\bar y)+\delta^{\frac{3n}{2}-1} r^{1-\frac{n}{2}}\left(\int_{M_{\frac{r}{2}}}|\nabla f|^2\right)^{\frac{1}{2}}\right)\\
    &\leq C(n,K)^{4\ell}\left(f( y)+\delta^{\frac{3n}{2}-1}r^{1-\frac{n}{2}}\left(\int_{M_{\frac{r}{2}}}|\nabla f|^2\right)^{\frac{1}{2}}\right).
\end{align*}

\end{proof}

\begin{lemma}[Lemma 1.8 in \cite{LiTam95}]
Let $M$ be complete non-compact and non-parabolic, $G(x,y)$ its minimal positive Green's function and $g(x)=-G(p,x)$. Then, for all $R>r>0$
\[
\ \int_{B_{R}(p)\cap M_{r}}|\nabla g|^{2}\leq 4(S(R)-i^{*}(r)),
\]
and
\[
\ \int_{B_{R}(p)\setminus B_{r}(p)}|\nabla g|^{2}\leq 4(S(R)-i(r)),
\]
where $S(R)=\sup_{\partial B_{R}(p)}g$, $i^{*}(r)=\inf_{\partial M_{r}}g$, $i(r)=\inf_{\partial B_{p}(r)}g$.
\end{lemma}

Thus
\[
\ \int_{B_{R}\cap M_{r}}|\nabla G(p,y)|^{2}d\mathrm{vol}(y)\leq 4\left(\sup_{\partial M_{r}}G-\inf_{\partial B_{R}(p)}G\right).
\]
Since $\inf_{\partial B_{R}(p)}G\to 0$ as $R\to\infty$, we get
\begin{equation}\label{e:energy}
\int_{M_{r}}|\nabla G(p,y)|^{2}d\mathrm{vol}(y)\leq 4\sup_{\partial M_{r}}G(p,y).
\end{equation}

As a consequence of the previous lemmas, we can prove the following estimate for the Green's function.
\begin{proposition}\label{th:GreenEstimatesMr}
Suppose that $M$ satisfies assumptions \eqref{e:euclvol} and \eqref{e:Ric delta}.
For any $x\in M_{r}$,
\[
G(x,p)\leq C(n,K,\beta)\delta^{3n-2}r^{2-n}.
\]
\end{proposition}
\begin{proof}
Suppose first that either $r<\frac{\delta}{12}$ or $r>3\delta$.
Set $f(x)=G(x,p)$. Because of the volume assumption \eqref{e:euclvol} and the curvature assumption \eqref{e:Ric delta}, the unbounded connected component $M_r^x$ of $M_r$ containing $x$ is non-parabolic.
Hence, we can find $y_{k}\to +\infty$ a sequence of points in $M_r^x$ such that $f(y_k)\to 0$ as $k\to\infty$. Note that either $2r<\frac{\delta}{6}$ or $2r>6\delta$. Hence, we can apply Lemma \ref{LemmaA.IX} with $2r$ replacing $r$, and $y=y_k$. Taking the limit as $k\to \infty$ gives, for all $\,x\in M_{r}$, \[
\ G(x,p)\leq C(n,v_0,V_0,K)\delta^{\frac{3}{2}n-1}r^{1-\frac{n}{2}}\left(\int_{M_{r}}|\nabla G|^{2}\right)^{\frac{1}{2}}.
\]
By continuity the same estimate holds also on $\partial M_r$.
Since $G(\cdot,p)$ is obtained as the limit of the increasing sequence of Green's functions relative to an exhaustion of $M$ with Dirichlet boundary conditions, by the maximum principle, $\sup_{M_{r}} G(\cdot,p)=\sup_{\partial M_{r}} G(\cdot,p)$. Hence we can choose $\bar x\in \partial M_r$ such that $f(\bar x)=\sup_{M_r} f$. By \eqref{e:energy},
\[
G(\bar x,p)\leq 
C(n,v_0,V_0,K)\delta^{\frac{3}{2}n-1}r^{1-\frac{n}{2}}
G(\bar x,p)^{\frac{1}{2}},
\]
so that, for any $x\in M_r$,
\[
G(x,p)\le G(\bar x,p)\leq 
C(n,v_0,V_0,K)
\delta^{3n-2}r^{2-n}.
\]
Finally, assume that $\frac{\delta}{12}<r<3\delta$. Since $\frac{r}{36}<\frac{\delta}{12}$, we can exploit the previous case and compute
\[
\ \sup_{x\in M_{r}}G(x,p)\leq \sup_{x\in M_{\frac{r}{36}}}G(x,p)\leq C(n,v_0,V_0,K)\delta^{3n-2} r^{2-n}.
\]
\end{proof}
We are now ready to prove Theorem \ref{th:green estimate_abstract}.
Suppose that $x\in \partial B_{r}(p)$, $r\leq 1$. Suppose $x\notin \overline{M_{r}}$. Let $y\in\partial M_{r}$.  We claim that
\begin{equation}\label{e:harnack final}
    \ C(n)^{-1} G(p,x)\leq G(p,y)\leq C(n)G(p,x).
\end{equation}
Let $\gamma$ be a minimizing geodesic from $x$ to $y$, then $\gamma$ has length at most $2r$.
Let $\sigma$ be a geodesic from $p$ to $y$, whose length is $l(\sigma)=d(y,p)$. The support of $\sigma$ can be covered by $9$ balls of radius $\frac{r}{8}$, say $B_{\frac{r}{8}}(\sigma(j\frac{r}{8}))$, $j=0,\ldots, 8$. If $x\in B_{\frac{r}{4}}(\sigma(j\frac{r}{8}))$ for some $j$ then $j=7,8$ (since otherwise $d(x,p)<r$). If $x\in B_{\frac{r}{4}}(y)$ or $B_{\frac{r}{4}}(\sigma(\frac{7}{8}r))$, then $G(p, \cdot)$ is harmonic in $B_{\frac{r}{2}}(y)\cup B_{\frac{r}{2}}(\sigma(\frac{7}{8}r))$ and
\eqref{e:harnack final} holds true
by Lemma \ref{lemma:harnack} applied (possibly twice) with $R=\frac{r}{2}\leq\frac{1}{2}$, $x_{0}=\sigma(\frac{j}{8}r)$, $j=7,8$, $\theta = K$.

If $x\notin \bigcup_{j=0}^{8}B_{\frac{r}{4}}(\sigma(j\frac{r}{4}))$, then $G(x,\cdot)$ is harmonic in $B_{\frac{r}{4}}(\sigma(j\frac{r}{4}))$ for any $j=0,\ldots, 8$, so that 
\begin{equation}    \label{e:harnack 2}
C(n)^{-1}G(x,y)\le G(p,x)=G(x,p)\leq C(n)G(x,y),
\end{equation}
again by 
Lemma \ref{lemma:harnack} applied nine times
with $R=\frac{r}{4}\leq\frac{1}{4}$, $x_{0}=\sigma(\frac{j}{8}r)$, $j=0,\dots,8$, $\theta = K$. Moreover in this case $y\notin B_{\frac{r}{4}}(x)$. Let $\eta:[0,r]\to M$ a minimizing geodesic from $p$ to $x$ and consider $B_{\frac{r}{4}}(\sigma(\frac{j}{8}r))$, $j=0,\ldots, 8$. As above, if $y\in B_{\frac{r}{4}}(\eta(j\frac{r}{8}))$ for some $j=0,\ldots, 8$ (necessarily $j=7$) and we can trivially conclude. Otherwise, \[
C(n)^{-1}G(y,x)\le G(p,y)=G(y,p)\leq C(n)G(y,x),
\]
 by Lemma \ref{lemma:harnack} applied again nine times, here to the function $G(\cdot, y)$ harmonic on $B_{\frac{r}{4}}(\eta(\frac{j}{8}r))$, $\forall\,j=0,\ldots, 8$, $\theta\leq K$. This latter together with \eqref{e:harnack 2} prove \eqref{e:harnack final}. 
 
 Now, applying Proposition \ref{th:GreenEstimatesMr} gives \eqref{e:Green abstract} for all $x\in\partial B_r(p)$.
Since $G(p,\cdot)$ is obtained as the limit of the increasing sequence of Green's functions relative to an exhaustion of $M$ with Dirichlet boundary conditions, by the maximum principle we obtain
\[
\ \sup_{x\in M\setminus B_{r}(p)}G(p,x)\leq \sup_{\partial B_{r}(p)}G(p,x)\leq C(n,K,v_0,V_0)\max\left\{1,d(o,p)^{3n-2}\right\}r^{2-n}.
\]
The proof of Theorem \ref{th:green estimate_abstract} when $\delta\ge 1$ is thus concluded.
\medskip

Finally suppose that $\delta<1$. Clearly
\[
2+2(d(p,x)-\delta)^2\ge 1 + (d(p,x)-1)^2,
\]
as this latter is equivalent to
\[
d(p,x)^2+2\delta^2-4d(p,x)\delta+2d(p,x)=\delta^2+ (d(p,x)-\delta)^2 + 2(1-\delta)d(p,x)\ge 0.\]
Then,
\begin{equation*}\mathrm{Ric}\geq -\frac{(n-1)K}{1+(d(p,x)-\delta)^2}\geq -\frac{(n-1)2K}{1+(d(p,x)-1)^2},\qquad\forall\, x\in M,\end{equation*}
and up to double the value of $K$ we can reduce to the case $\delta\ge 1$.

\section*{Data availability}
Data sharing not applicable to this article as no datasets were generated or analyzed during
the current study.

\section*{Conflict of interest statement}
On behalf of all authors, the corresponding author states that there is no conflict of interest.

\bibliographystyle{alpha}
\bibliography{ABP}
\end{document}